\documentclass[12pt]{amsart}
\usepackage[left=1in,top=1in,right=1in,bottom=1in]{geometry}
\usepackage{mathrsfs}
\usepackage[english]{babel}
\usepackage{amsmath,amsthm,amsfonts,amssymb,epsfig,xcolor, hyperref}
\usepackage{bbm,booktabs,float,mathtools,siunitx,tikz}
\usepackage[shortlabels]{enumitem}
\usepackage{breqn}
\usepackage{stmaryrd}
\usepackage{relsize}
﻿\usepackage{cancel}
\def\C{\ensuremath\mathbb{C}}

\def\Z{\ensuremath\mathbb{Z}}

\def\P{\ensuremath\mathbb{P}}

\def\E{\ensuremath\mathbb{E}}

\numberwithin{equation}{section}
\newtheorem{thm}{Theorem}
\newtheorem{lem}[thm]{Lemma}
\newtheorem{prop}[thm]{Proposition}
\newtheorem{cor}[thm]{Corollary}
\newtheorem{df}[thm]{Definition}
\newtheorem*{rem}{Remark}

\numberwithin{thm}{section}
﻿

﻿

﻿

﻿

﻿

﻿
﻿
﻿
﻿
﻿
﻿
﻿
﻿
﻿
\begin{document}

	\title[Upper tail distributions  of central $L$-values of elliptic curves]{Upper tail distributions  of central $L$-values of  quadratic twists of elliptic curves at the variance scale}
	\author{Nathan Creighton}
	\address{N. Creighton, Mathematical Institute, University of Oxford, UK}
	\email{creighton@maths.ox.ac.uk}
	\maketitle	
	\begin{abstract}
		We consider the large deviations at the order of the variance for the central value of a family of $L$-functions among the members with bounded discriminant. When there is an upper bound on an integer moment of the central value twisted by a short Dirichlet polynomial,  we can establish upper bounds on the density of members exhibiting a large central value. We adapt the techniques from Arguin and Bailey for large deviations of the Riemann zeta function to prove results on the degree two family of quadratic twists of an elliptic curve. This upper bound improves on density results previously obtained by Radziwi\l{\l} and Soundararajan.
\end{abstract}

\section{Introduction}
A key area of interest in the study of families of $L$-functions is their moments. The Lindel\"{o}f Hypothesis, which has implications for zero density estimates of $L$-functions, may be stated in terms of the moments possessing subpolynomial growth. Usually, the length of the Dirichlet polynomials required to approximate powers of $L$-functions obstruct taking large moments and most non-integer moments accurately; leading-order asymptotics for moments of the Riemann zeta function are currently only known for the second \cite{HL}, and fourth \cite{Ing}. Even for the sixth moment, it is unknown whether there there is subpolynomial growth. The problem of bounding higher fractional moments $I_k(T)$ of the Riemann zeta function is even less well understood, and current results rely on interpolating with bounds on a higher moment, such as the twelfth, due to Heath-Brown \cite{HB78}. Assuming RH, much more is known, and the refinement of Harper \cite{Harper} to Soundararajan's scheme \cite{Sound} gives upper bounds of the correct size for $I_k(T)$ for $k\ge 1$, and lower bounds of the same order are found in \cite{RSlower} and \cite{HSlower}. This extended the work of Radziwi\l{\l} in \cite{4.36}, which gave upper bounds for $I_k(T)$ of the same order,
 for all $k<2+\frac{2}{11}.$ In \cite{KS}, the authors conjectured that the moments of the Riemann zeta function should behave like the moments of the characteristic polynomial of a unitary matrix, and were thus able to predict asymptotics for  $I_k(T)$, for all $k>-\frac{1}{2}.$ \\
The following Central Limit Theorem due to Selberg \cite{Sel} provides the distribution of $\log\zeta\left(\frac{1}{2}+it\right)$ at the level of the standard deviation.
\begin{thm}[Selberg's Central Limit Theorem]\label{CLT} Let $E\subset\C$ be any measurable set. Then \begin{align}\label{SCLT}
		\lim_{T\to \infty}\frac{1}{T}\left|\left\{t\in [T,2T]: \frac{\log \zeta\left(\frac{1}{2}+it\right)}{\sqrt{\frac{1}{2}\log\log T}}\in E\right\}\right|=\frac{1}{2\pi}\iint_E e^{-\frac{x^2+y^2}{2}}dxdy.
\end{align} \end{thm}
However, to understand the moments of the Riemann zeta function, it is the distribution at the level of variance which must be understood, and in \cite{AB} they study this to give an upper bound on $I_k(T)$ for all $0\le k\le 2.$  

Besides that of the Riemann zeta function, there has been much work on the moments of other families of $L$-functions. In \cite{KzS}, the values of a family of $L$-functions is conjecturally associated to the characteristic polynomial of a classical compact group of unitary matrices. In \cite{CF}, by calculating the moments over the group of matrices conjectures are produced for the moments over a family of $L$-functions with bounded conductor at a fixed point. More precisely, if $\mathscr{F}$ is a family of $L$-functions $L_f$ with central value $\frac{1}{2},$ they construct sample spaces of families with bounded \textit{conductor} $c(f).$  This leads to a conjecture for the moments of the shape
\begin{equation}\label{momconjgen}
	\frac{1}{Q^\ast} \sum_{\substack{f\in \mathscr{F}\\
		c(f)\le Q}}V\left(L_f\left(\frac{1}{2}\right)\right)^k\sim \frac{a_kg_k}{\Gamma (1+B(k))}\log ^{B(k)}\left(Q^A\right). 
\end{equation}
Here $V(z)$ is some measure of the size of $z$, and so they put $V(z)=z$ or $|z|^2 $ depending on whether the values are positive or complex, and $a_k,g_k, A$ and $B(k)$ are parameters depending on the functional equation, symmetry group and the family of the $L$-function.
 
 In \cite{KS2} they conjecture that when the family depends on quadratic twists, then the logarithm of the $L$-function is associated to the characteristic polynomial of an orthogonal family of matrices. By studying the moments of the characteristic polynomials in the Gaussian Orthogonal Ensemble, they were able to produce an associated conjecture for the moments of the $L$-function. The logarithms of specific $L$-functions exhibiting orthogonal and symplectic symmetries were studied in \cite{Hough14}, extending the typical study of the Keating-Snaith conjectures from the unitary setting.

The orthogonality properties for any family produced by quadratic twists of an $L$-function all emanate from the orthogonality properties of quadratic characters. This has been studied by \cite{GZ24} in producing moments for the central values of quadratic Dirichlet $L$-functions. The general study of moments of a family of quadratic twists of an  $L$-function is similar; the terms twisted by a square remain similar across all choices of discriminant, so should be thought of as the diagonal terms, while the terms twisted by a non-square should exhibit a high level of cancellation when averaging across the discriminant, so should be viewed as the off-diagonal terms.

The fractional moments of a random variable with a log-normal distribution, which many families of $L$-functions are predicted in \cite{KS2} to possess, are controlled by their large deviations at the order of their variance. In the unitary setting, the large deviations of a characteristic polynomial were considered in \cite{hko}, for deviations ranging in order from the standard deviation to the pointwise maximum. Analogous results might be predicted to hold for the logarithm of $L$-functions, at least for ranges up to the variance.  In \cite{AB}, the large deviations of the Riemann zeta function high up the critical line are bounded above by viewing the logarithm  as a random walk on primes, and applying the barrier method. This provides an upper bound on all fractional moments of the  Riemann zeta function up the second, which is believed to be sharp up to a constant. Unlike the results from \cite{Sound} and \cite{Harper}, which assumed RH, their results were unconditional.

Random walks can be used to model the logarithms of many different families of values of $L$-functions. In \cite{AC25}, the methods were adapted to bound the large deviations of central values of Dirichlet $L$-functions with fixed large modulus. Bounding moments and large deviations of an $L$-function becomes trickier as the degree increases. Here, we give the first example of the application of the barrier method to the context of a degree $2$ family of $L$-functions, the quadratic twists $L(s,E_d)$ of the $L$-function associated to an elliptic 
curve, $E$. We put $N$ to be the conductor of $E$, and $N_0=\text{lcm}(N,8).$ Then we can consider the distribution of the central values of these $L$-functions, as we range over twists with $|d|\le X$, for some large parameter $X$. Each such $L$-function has a Dirichlet series
\begin{equation}
	L(s,E_d)=\sum_n \frac{a_n\chi_d(n)}{n^s},
\end{equation}
 an Euler product whose factor at  a prime $p$ depends on whether there is good or bad reduction, and a completed $L$-function

\begin{equation}
	\Lambda(s,E_d)=\left(\frac{\sqrt{N}|d|}{2\pi}\right)^s\Gamma\left(s+\frac{1}{2}\right) L\left(s,E_d\right).
\end{equation}
Here the coefficients of $L\left(s,E_d\right)$ are normalised with the Hasse bound reading $|a(p)|\le 2$, so that the central value lies at $\frac{1}{2}.$ This $L$-function obeys the functional equation
\begin{equation}\label{functeqn}
	\Lambda(s,E_d)= \epsilon_E(d)\Lambda(1-s,E_d),
\end{equation}
for some \textit{root number} $\epsilon_E(d)\in \{\pm 1\}.$ 
Moreover, the value of $\epsilon_E(d)$ only depends on the value of $d$ mod $N_0,$ since $\epsilon_E(d)=\epsilon_E \chi_d(-N).$ 
\subsection{Main results} We want to consider how often the central $L$-value takes large values.
If $\epsilon_E(d)=-1,$ then the odd symmetry in Equation \eqref{functeqn} means the central value vanishes, and so to study an interesting distribution of central values, we restrict to the discriminants with $\epsilon_E(d)=+1.$
That is, the set of discriminants $\mathcal{E}$ to consider is the fundamental discriminants coprime to $2N,$ with \textit{root number} $\epsilon_E(d)=+1.$ 

Theorem 1 in \cite{RS2} states that for $0\le \alpha\le 1,$

\begin{equation}\label{olddev}
	\left|\left\{d\in \mathcal{E}, |d|\le X: \log\left(L\left(\frac{1}{2}, E_d\right)\right)\ge \left(\alpha-\frac{1}{2}\right)\log\log|d|\right\}\right|\ll X(\log X)^{-\frac{\alpha^2}{2}}.
\end{equation}
This is off from the expected Gaussian distribution by a factor of the standard deviation, $\sqrt{\log\log X},$ and the rest of this paper will be devoted to retrieving the predicted tail.

 By quoting results from \cite{RS2}, we can impose certain arithmetic properties on the set of discriminants considered. We can consider the distribution of central values restricted to discriminants $d$ that are multiples of a given natural number $v$ with $(v,N_0)=1,$ and have sign $\mathscr{O}$. Moreover, we can restrict to a given congruence class $a\in \left(\Z/N_0\Z\right)^\ast,$ with $a\equiv 1 \text{ mod 4}$. This leads us to define the set of permissible discriminants for such a choice of parameters:
\begin{equation}\label{SS}
	\mathcal{E}(\mathscr{O},a,v)=\left\{d\in \mathcal{E}: \mathscr{O} d>0,\ v|d,\ d\equiv a \text{ mod }N_0\right\}.
\end{equation} We prove the following theorem, which generalises the distribution to members of $	\mathcal{E}(\mathscr{O},a,v)$.\newpage
 \begin{thm} \label{largedevthm} Let $v$ be a fixed natural number with $(v,N_0)=1, \mathscr{O}\in\{\pm 1\}$ be a fixed sign, $a\in \left(\Z/N_0\Z\right)^\ast$ be a fixed congruence class with $a\equiv 1$ mod $4$ such that $\epsilon_E(a)=+1,$ and $0<\alpha<\frac{1}{2}$. If $V\sim \alpha \log\log X$,  then for $X$ sufficiently large, we have
	\begin{equation}\label{eqlargedev}
		\left|\left\{d\in \mathcal{E}(\mathscr{O},a,v):\ |d|\le X,\ \log\left(L\left(\frac{1}{2}, E_d\right)\right)\ge V-\frac{\log\log X}{2}\right\}\right|\ll \frac{X}{vN_0}\int^\infty_V\frac{e^{-\frac{y^2}{2\log\log X}}}{\sqrt{\log\log X}}dy,
	\end{equation}
	where the implicit constant is uniform for $\alpha$ in $(0,B),$ for any $B<\frac{1}{2}$.

\end{thm}
\begin{rem}
	The method of proof used requires modification to be valid for the case $\alpha=0,$ where $V=o(\log\log X).$ The barrier method behaves differently when the deviation is not of the same scale as the variance, and we do not consider this range.
\end{rem}
By expressing moments in terms of the measure of high points as in Section 3.1 of \cite{AB} and using dyadic dissection, we deduce a conjecturally sharp bound on fractional moments of the central value, which agrees with the bounds from Theorem 1 in \cite{RS2}.
\begin{cor}
	Let $0<\alpha<1.$ Then for $X$ sufficiently large, the fractional moments of the central value obey
	\begin{equation}
		\sum_{\substack{d\in \mathcal{E}(\mathscr{O},a,v)\\ |d|\le X}}L\left(\frac{1}{2}, E_d\right)^{\alpha}\ll \frac{X(\log X)^{\frac{\alpha^2-\alpha}{2}} }{vN_0},
	\end{equation}
	where the implicit constant is uniform for $\alpha$ in $(0,B),$ for any $B<1$.
\end{cor}
Note that if we take $v=1,V=\alpha\log\log X$ and sum over all choices of signs and congruences  in Theorem \ref{largedevthm}, then we improve on Equation \eqref{olddev} for the range $0<\alpha<\frac{1}{2}$ by the missing factor of the standard deviation, but having the more continuous setup of a large deviations result for $V\sim\alpha\log\log X$ gives greater flexibility, and contrasts to results dependent on the height or other arithmetic properties of $\alpha$.

\subsection{Proof method} Our proof follows the barrier methods used in \cite{AB24} and \cite{AC25} to control the logarithm of the $L$-values via partial sums. 
In order to deploy the barrier method, we need to break the primes contributing to the central value into different intervals, and study the behaviour of the primes between the endpoints of each interval. The first interval, $P_1$, will contain all the small primes. We need to take long moments to control the behaviour of the primes, and so correspondingly the primes must all be small to give a short Dirichlet polynomial. We take $$X_1=X^{\frac{1}{2\lceil(\log\log X)^2\rceil}}$$ to be the first step. Having accounted for the small primes in $P_1,$ we can then split up the larger primes into intervals. For $j\ge 2,$ we put $$l_j=2\lceil\log^{\mathbf{s}}_{j+1}(X)\rceil,$$ where
\begin{equation}\label{s}
	\mathbf{s}=\frac{10^5}{1-2\alpha},
\end{equation}  and then consider the steps $$X_j=X^{-l_j},$$ which are analogous to the time-steps $T_l$ in \cite{AB} and $q_l$ in \cite{AC25}.  Our last step will be $X_R,$ where $R$ is the largest integer satisfying $\log_{R+2}X>10^5-\log\alpha.$ The $\log\alpha$ term is necessary to ensure adequate spacing between the time steps for Equation \eqref{decompH}, when the gradient, $\kappa$, is small. For convenience in later equations, we put $$l_1=2\lceil(\log\log X)^2\rceil,\quad \text{ and }l_0=\left(2\lceil(\log\log X)^2\rceil\right)^{10^{-5}}+l_1.$$

The choice of $\mathbf{s}$ in Equation \eqref{s} will be required to prove the necessary bounds for Theorem \ref{largedevthm}, but broadly speaking, using a twisted mollifier for the first moment and Markov's inequality in the barrier method for a walk of length $\alpha<\frac{1}{2}$ requires an interval of length proportional to $\left(\frac{1}{2}-\alpha\right)^{-1}$ for the variance, $\log\log p.$

We consider the logarithm as split up into its contributions from the different intervals of primes separated by the $X_j$. Here, we define the partition of primes up to $X_R$ by setting $P_1$ to be the primes  up to $X_1,$ and for $2\le j\le R,$ setting \begin{equation}
P_j=\{p\in (X_{j-1},X_j]\}.\end{equation}

\subsection{Choice of mollifier}\label{molchoice}	
In order to bound the distribution of the central value of the $L$-functions, we need to show the  primes in different intervals $P_j$ act independently. This means that for $1\le r<R,$ the contribution of the primes $p\le X_r$ to the central value should be viewed as independent of the contribution from the primes in the interval $P_{r+1}.$ We want a mollifier $M_r$ of the contribution of the primes $p\le X_r$ to the central value $L\left(\frac{1}{2},E_d\right).$ Exploiting the idea of independence, this mollifier should factor as a product $M_r=\prod_{j\le r} A_j$, with $A_j$ a mollifier of the contribution of the primes in $P_j$. If the mollification is successful in cancelling this contribution, then $M_r(d)L\left(\frac{1}{2},E_d\right)$ should be influenced only by primes $p>X_r,$ so should be only weakly dependent on the walk on the primes $p\le X_r$.\\
In the $q$-aspect in \cite{AC25} and the $t$-aspect in \cite{ABR20}, the mollifier is given by some truncated inversion of the formal Euler product. Ignoring the finitely many primes of bad reduction for $E$, which will all lie in $P_1$ if $X$ is sufficiently large, this would suggest a truncated mollifier approximating the formal reciprocal:

\begin{equation}
\mathcal{M}_r(d)=\prod_{p\le X_r \text{  good} } 1-a(p)\chi_d(p)p^{-\frac{1}{2}}+\chi_d(p)^2p^{-1}.\end{equation}
 In order to use Markov's inequality, we require a non-negative mollifier. In the $t$-aspect in \cite{AB} and the $q$-aspect in \cite{AC25}, we take even moments of the $L$-function, and hence this is ensured by taking even moments of the mollifier, however Proposition $2$ in \cite{RS2} only concerns first twisted moments.  Although the central value (and hence its reciprocal) is always positive by Waldspurger's Theorem, when we truncate there is no guarantee of the truncation being non-negative, and so we must take an alternative approach to ensure non-negativity.

A further difference in constructing the mollifier comes from observing that the square term $\chi_d(p)^2$ is very predictable; it is $1$ if $p\nmid d$ and $0$ otherwise. Whereas in the non-quadratic scenarios, the effect of the squares of primes had to be incorporated into the series approximating the logarithm, we find the model simplifies in this quadratic case. We would like a simple approximation to $\log  \mathcal{M}_r(d)$, to model as a random walk. If we take the simple partial sum $\mathscr{P}_j(d)$, defined as
\begin{equation}\label{defpj}\mathscr{P}_j(d):=\sum_{p\in P_j} \frac{a(p)\chi_d(p)}{p^{\frac{1}{2}}},\end{equation} then we would have
\begin{equation}\label{fullexp}
\exp\left(-\sum_{j\le r}\mathscr{P}_j(d)\right)=\mathcal{M}_r(d) \exp\left(\sum_{p\le X_r} p^{-1} \left(\chi_d(p)^2-\frac{a(p)^2}{2}\right)+O(p^{-\frac{3}{2}})\right).  
\end{equation} 
The sum over the $O(p^{-\frac{3}{2}})$   term is uniformly bounded, while the 
term associated to the squares, $$\exp\left(\sum_{p\le X_r} p^{-1} \left(\chi_d(p)^2-\frac{a(p)^2}{2}\right)\right),$$ should be close for all different values of $d\in \mathcal{E}(\mathscr{O},a,v).$ If we used this as a mollifier, then we would expect $\exp\left(-\sum_{j\le r} \mathscr{P}_j(d)\right)L\left(\frac{1}{2},E_d\right)$ to be approximately proportional to $ \mathcal{M}_r(d)L\left(\frac{1}{2},E_d\right) $, and hence weakly dependent on the behaviour of the primes $p\le X_r$.
In order to give a short enough Dirichlet polynomial to take twisted moments, we take as the mollifier a truncation of the Taylor expansion for the exponential given in Equation \eqref{fullexp}.
The quadratic twists at primes in the interval $\mathscr{P}_1$, which include any fixed prime for sufficiently large $X$, do not behave like independent Rademacher random variables. As in \cite{RS2}, in order to control their contribution to the central value, we take a long truncation of the exponential, 
\begin{equation}\label{A1}
	\mathscr{A}_1(d)=\sum^{20\lceil \log\log X\rceil}_{r=0} \frac{(-\mathscr{P}_1)^r}{r!}.
\end{equation}
For $j>1$, we cannot take such a long truncation of the exponential in the construction of the mollifier, since the length must be  bounded by a small power of $X$ to take moments. However, as in the $t$-aspect in \cite{AB}, we will see a shorter mollifier suffices as the twists $\chi_d(p)$ behave closer to independent random variables, and thus take
\begin{equation}\label{defaj}
	\mathscr{A}_j(d)=\sum^{(l_j-l_{j+1})^{10^5}}_{r=0} \frac{(-\mathscr{P}_j)^r}{r!}.
\end{equation}
Here, $(l_j-l_{j+1})^{10^5}$ is a large even integer, and so Lemma 1 in \cite{RS2} guarantees that the mollifier $\mathscr{A}_j(d)$ is always positive for $1\le j\le R.$ We then define the mollifier of the primes $p\le X_r$ to be  \begin{equation}\label{Mollchoice}
	M_r(d)=\prod_{j\le r}\mathscr{A}_j(d).
\end{equation}	We will encode information pertaining to bounds on the sums $\mathscr{P}_j(d)$ through \textit{twists}, which will need to be short enough to take twisted mollified moments. Since we expect the walks in disjoint intervals to behave independently, we may assume the twists  split into short factors determined by primes in separate intervals $P_j$,  with coefficients varying depending on $d$. This motivates the following definition of twists, which we will use in Proposition \ref{twistmollprop}, our bound for twisted mollified moments. \begin{df}[Well-factorable twists]\label{Qdef}

Say $Q_d$ is degree $r$ well-factorable if we can write

\begin{equation}\label{Qsplit}
	Q_d(s)=\prod_{1\le j \le r} Q_{d,j}(s),\end{equation}  where \begin{equation}
	Q_{d,j}(s)=\sum_{\substack{p|m \implies p\in P_{j}\\ \Omega_j(m)\le 10(l_j-l_{j+1})^{10^4}}}\frac{\chi_d(m)\gamma(m)}{m^s},
\end{equation}
and $\gamma(m)$ are arbitrary real coefficients.\end{df} Note that any such  well-factorable polynomial of degree $r\le R$ has length at most
\begin{equation}\label{Tbound}
	\prod^{R}_{j=1} X_j^{10(l_j-l_{j+1})^{10^4}}=\exp\left(\sum^{R}_{j=1} 10(l_j-l_{j+1})^{10^4}\log X_{j}\right)\ll X^{\frac{1}{1000}}.
\end{equation}
The mollifier $M_r(d)$  has length at most \begin{equation}
	\prod^r_{j=1} X_j^{(l_j-l_{j+1})^{10^5}}=\prod^r_{j=1}X^{\frac{(l_j-l_{j+1})^{10^5}}{l_j}} \ll X^{\frac{1}{1000}},
\end{equation} so that the twisted mollifier $M_r(d)Q_d\left(\frac{1}{2}\right)^2$ has length $\ll X^{\frac{1}{100}}.$
This enables us to apply Proposition 2 from \cite{RS2} to take the first mollified moment of the central value, multiplied by the well-factorable twist $Q_d,$ with negligible error terms.

The following proposition gives a twisted mollifier formula for the central value of the twist $Q_d(\frac{1}{2})$. For convenience, we write the expansion of the product $Q_d(s)=\prod_{1\le j \le r} Q_{d,j}(s)$ as \begin{equation}\label{defQd}
	Q_d(s)=\sum^N_{\substack{w=1\\p|w\implies p\le X_r}} \frac{C(w) \chi_d(w)}{w^s},
\end{equation}
where the $C(w)$ denote the real coefficients in $Q_d(s),$ and the length satisfies \begin{equation}
	N\ll X^{\frac{1}{1000}}.\end{equation}
From now on, we will drop the implicit dependence of the $L$-function, mollifier and twist  on $d$, and assume we take the central value, so for example write $Q$ for $Q_d	\left(\frac{1}{2}\right).$
\begin{prop}\label{twistmollprop}
	Let $1\le r \le R$, and let the mollifier $M_r$ be as in Equation \eqref{Mollchoice} and the degree $r$ well-factorable twist $Q$ be as in Definition \eqref{Qdef}. Then \begin{equation}\E\left[LM_rQ^2\right]\ll \log ^{-\frac{1}{2}}(X_r)\sum_{\substack{p|q\implies p\le X_r\\ q \text{ square-free}}}\left|\sum_{\substack{p|u_1u_2\implies p\le X_r\\ u_1q,u_2q=\square}}
		\frac{C(u_1)C(u_2)}{(u_1u_2)^{\frac{1}{2}}}\prod_{p|u_1u_2}\left(1-\frac{1}{p}\right)\right|.\end{equation}
\end{prop}

We remark that the terms $u_1q=u_2q=\square$, where we have borrowed the notation from \cite{RS2} to denote those pairs where $u_1$ and $u_2$ have the same square-free part, $q$, are the diagonal terms. 	Unlike the $q$-aspect, the diagonal terms are not just those cross-terms with the same twist; in the context of quadratic characters, the cross-terms whose twists have the same square-free part. In order to get an accurate bound on cancellation in the diagonal terms, we must use the triangle inequality only when summing over the square-free parts, not internally within the diagonal terms. Extra care must be taken to preserve this cancellation throughout the estimates. 
 We postpone the proof of Proposition \ref{twistmollprop} to Section \ref{SecTM}, and proceed with  the proof of Theorem \ref{largedevthm} in Section \ref{SecLD}. The proofs in Section \ref{SecLD} require bounds on the moments of Dirichlet polynomials averaged over quadratic twists provided in Appendix \ref{Sec: moments}, while to simplify the calculations for the proof of Propositions \ref{twistmollprop}, we appeal to lemmata on quadratic forms proven in Appendix \ref{Sec: QF}.
 \subsection{Notation}
 In order to use probabilistic tools such as the barrier method, we define a sample space to be the uniform distribution over members of $\mathcal{E}(\mathscr{O},a,v)$ with magnitude at most $X$, then take probabilities and expectations with regards to this sample space. Thus we may express the size of the set of twists with large deviations in Equation \eqref{eqlargedev} in terms of probabilities with regards to the counting measure.

\subsection{Acknowledgements} The author would like to thank Louis-Pierre Arguin for suggesting the problem, edits to the paper  and for many helpful discussions throughout the project, and his supervisor Jon Keating for his corrections to an earlier draft and guidance.  The author is supported by the EPRSC grant EP/W524311/1.  
\section{Proof of Theorem \ref{largedevthm}}\label{SecLD}The bound we have to prove for Equation \eqref{eqlargedev} may be expressed as saying that for $X$ sufficiently large, we have
\begin{equation}\label{countbig}
	\P\left(  \log\left(L\left(\frac{1}{2}, E_d\right)\right)\ge V-\frac{\log\log X}{2}\right)\ll \frac{e^{-\frac{V^2}{2\log\log X}}}{\alpha\sqrt{\log\log X}},
\end{equation} 
where the implicit constant is uniform for $\alpha$ in $(0,B)$ for any $B<\frac{1}{2}.$

Assuming Proposition \ref{twistmollprop}, we may proceed with the proof of the above equation. Our method follows the recursive scheme in \cite{AB} and \cite{AC25} used to view the logarithm as a walk on the primes, but simplifies because the Dirichlet series are real, so there is no need to consider the imaginary part. Moreover, since the twists are quadratic, we no longer need to incorporate the effects of the squares of primes into the partial sums. 
We divide the set of primes into intervals $P_r$ and then define the partial sums at each time step:

\begin{equation} \label{subwalk} S_{r}=\sum_{p< X_r} \frac{\chi_d(p)a(p)}{p^{\frac{1}{2}}},\end{equation}
for $0\le r\le R.$ 
Note that we can partition each partial sum into the contributions from each interval $P_j$ for $j\le r$, so that \begin{equation}
	S_r=\sum_{j\le r}\mathscr{P}_j,
\end{equation} where $\mathscr{P}_j$ is the increment over the interval $P_j$ defined in Equation \eqref{defpj}. We view $S_r$ as a random walk, by thinking of the discriminant $d$ as a random variable. We then have to find the probability of a large central value, given by the event \begin{equation}
	H:=\left\{\log L\left(\frac{1}{2},E_d\right)\ge V-\frac{\log\log |d|}{2}\right\}.
\end{equation} 

For convenience, we define the approximate  variance of the walk $S_r:$ \begin{equation}
	n_r:=\log\log \left(\max\{X_r,e\}\right).
\end{equation}  Inspired by the typical behaviour of random walks, the method involves constructing a barrier at each time step $\log \log X_r$ which the sub-walks $S_{r}$ for members of $H$ should typically lie in. 
The average height $S_r$  at time $n_r$  in $H$  is modelled by linear growth  at rate $\kappa,$ where \begin{equation}\label{kappa}
	\kappa=\frac{V}{\log\log X}.
\end{equation}
Observe that by the definition of $V$, we have $\kappa \to \alpha$ as $X\to\infty$. For $1\le r\le R,$ we define  lower and upper barriers $L_r$ and  $U_r$ for the random walk, and show that, conditional on $H,$ there is a very small probability of the walk being outside the barrier at time $n_r$. The logarithm of the central value is typically influenced by primes up to about $X$, and so the variance of the section of the logarithm influenced by primes not mollified by $M_r$ is approximately \begin{equation}
	\sigma^2_r:=\log\log X-n_r.
\end{equation} 
By the construction of $l_r,$ we see that $\sigma^2_r=\log l_r.$ To use the barrier method, we set the lower and upper barriers at time $n_r$ to be \begin{equation} L_r=\kappa n_r-\mathbf{s}\sigma^2_{r},\quad
U_r=\kappa n_r+\mathbf{s} \sigma^2_r,\end{equation} with $\mathbf{s}$ as defined in Equation \eqref{s} and $\kappa$ as defined in Equation \eqref{kappa}.

Note that this is where we expect the random walk $S_{n}$ to lie; if we consider the full truncated logarithm for $\log L(\frac{1}{2},E_d)^{-1},$ which ignoring the finitely many primes of bad reduction would look like $$\sum_{k\ge 1}\sum_{p< X_n}\frac{\chi_d(p^k)a(p)^k}{kp^{\frac{k}{2}}},$$ then we would expect this to be lower than $S_r$  by approximately $\frac{n_r}{2}$ at time $n_r$. This is why the distribution of $\log L(\frac{1}{2},E_d)$ is off-centred, with conjectured expectation $-\frac{\log\log X}{2}.$

Due to the simplifications of having real characters and excluding the effects of squares of primes from the random walk, we are able to take a much simpler recursive scheme than in \cite{AB24} or \cite{AC25}. \\
We take $A_0=B_0=\{d\in \mathcal{E}(\kappa,a,v): |d|\le X\}$ to be the full sample space, and for $1\le r\le R$ define the events of staying within the barrier:
\begin{equation}
	A_r=A_{r-1}\cap\{S_r<U_r \},\quad B_r=B_{r-1}\cap\{S_r>L_r\}.
\end{equation} We put $G_r=A_r\cap B_r$ for $0\le r \le R$, which is the event that $S_j$ lies between $L_j$ and $U_j$ for all $1\le j\le r.$ 

We decompose the large deviation event $H$ to express the probability as
\begin{equation}\label{decompH}
	\P(H)= \sum^{R-1}_{r=0}\P(H\cap G_r\cap G_{r+1}^c)+\P(H\cap G_{R}).
\end{equation}
We prove the following Proposition, which allows us to decompose walks with a large central value satisfying $H$ into walks with abnormal intermediary values, and then show each is improbable.
\begin{prop}\label{barrierprob}

\renewcommand{\theenumi}{\roman{enumi}}
 Let $B<\frac{1}{2}$. Then there exists $\delta>0$ such that for all $\alpha\in (0,B)$, with the event $H$ and parameter $V$ defined above,  we have for $0\le r\le R-1:$ \begin{enumerate}
	\item\begin{equation}\label{midprob}
		\P(H\cap G_r\cap G_{r+1}^c)\ll \frac{e^{-\frac{V^2}{\log\log X}}}{\alpha\sqrt{\log\log X}}e^{-\delta\kappa\mathbf{s}\sigma^2_{r+1}}.
	\end{equation}
	Moreover,
	\item \begin{equation}\label{endprob}
		\P(H\cap G_R)\ll \frac{e^{-\frac{V^2}{\log\log X}}}{\alpha\sqrt{\log\log X}}.
	\end{equation}
\end{enumerate}
\end{prop}
We recall $\log_{R+2}(x)>10^5-\log \alpha,$ then substitute Equations \eqref{midprob} and \eqref{endprob} into Equation \eqref{decompH} to show $\P(H)\ll \frac{e^{-\frac{V^2}{\log\log X}}}{\alpha\sqrt{\log\log X}}.$
 \subsection{Proof of Proposition \ref{barrierprob} (i)}
 In this section we will show there is a small probability of the walk leaving the barrier time $X_r$, conditional on staying in the barrier until time $X_{r-1}$. Being able to condition on staying in the barrier until time $X_{r-1}$ gives a shorter walk to consider, and ultimately a shorter twist, which enables us to control the walk up to a small power of $X.$ For convenience, we will take $S_{0}=X_0=0.$ Unlike in \cite{AB24} and \cite{AC25}, this avoids having to treat the event that the first barrier is breached, $P(H\cap G_1^c),$ as a special case.
Then we may write \begin{equation}\label{decompmid}\P(H\cap G_r\cap G_{r+1}^c)\le \P(G_{r}\cap A_{r+1}^c)+\P(H\cap G_r\cap B_{r+1}^c),\end{equation} and bound the probability of both events.

We first bound $\P(G_{r}\cap A_{r+1}^c)$ using Markov's inequality. To break the barrier at time $n_{r+1}$ conditional on $G_{r}$, we must have $S_{{r+1}}>U_{r+1},$ and so we see  we see that for any $k_r>0$: \begin{equation}\label{AMarkov}
	\P(G_{r}\cap A_{r+1}^c)\ll \sum_{u\in [L_r,U_r]} \E\left[\frac{\left|S_{{r+1}}-S_{r}\right|^{2k_r}}{(U_{r+1}-u)^{2k_r}}\mathbf{1}\left(G_r\cap \{S_{r}\in [u,u+1]\}\right) \right].
\end{equation}  
If we set $k_r=\left\lfloor\frac{ (U_{r+1}-u)^2 }{2\left(\sum_{X_{r}\le p\le X_{r+1}}\frac{a(p)^2}{p}\right)}\right\rfloor,$ then using Lemma \ref{walkmom} and Equation \ref{2.4} from Lemma \ref{pointmom}, we see this is \begin{equation}
\ll 	\sum_{u\in [L_r,U_r]}\frac{e^{-\frac{u^2}{2n_r}}}{\sqrt{n_r}} \frac{\frac{(2k_r)!}{2^{k_r}(k_r)!}\left(\sum_{X_{r}\le p\le X_{r+1}}\frac{a(p)^2}{p}\right)^{k_r}}{(U_{r+1}-u)^{2k_r}}.  \end{equation}
By Stirling's formula, this is

\begin{align}
	&\ll \sum_{u\in [L_r,U_r]}\frac{e^{-\frac{u^2}{2n_r}}}{\sqrt{n_r}} \left(\frac{2k_r}{e}\right)^{k_r}\frac{\left(\sum_{X_{r}\le p\le X_{r+1}}\frac{a(p)^2}{p}\right)^{k_r}}{(U_{r+1}-u)^{2k_r}} \\
	&\ll \sum_{u\in [L_r,U_r]}\frac{e^{-\frac{u^2}{2n_r}}}{\sqrt{n_r}}e^{-k_r}\\
	&\ll \sum_{u\in [L_r,U_r]}\frac{e^{-\frac{u^2}{2n_r}}}{\sqrt{n_r}}e^{-\frac{ (U_{r+1}-u)^2 }{2\left(\sum_{X_{r}\le p\le X_{r+1}}\frac{a(p)^2}{p}\right)}}.
\end{align}
Comparing this to the law for a sum of two independent Gaussian variables $Z_1\sim N(0,n_r)$ and $Z_2\sim N(0,n_{r+1}-n_r+o_{r\to\infty}(1)),$ the above expression may be bounded as

$\ll \P(Z_1+Z_2\ge U_{r+1})\ll \frac{ e^{-\frac{U_{r+1}^2}{2n_{r+1}}}}{\alpha\sqrt{ n_{r+1}}}.$ 

Since $U_{r+1}=\kappa n_{r+1}+\mathbf{s}\sigma^2_{r+1},$ this is
\begin{align}
	\nonumber&\ll \frac{e^{-\frac{\kappa^2 n_{r+1}}{2}-\kappa \mathbf{s} \sigma^2_{r+1}}}{\sqrt{\log\log X}}  \\
	&\ll \frac{e^{-\frac{\kappa^2 \log\log X}{2}}}{\alpha\sqrt{\log\log X}} e^{(\frac{\kappa^2}{2}-\kappa \mathbf{s})\sigma^2_{r+1}}. 
\end{align}
Since $\mathbf{s}>1>\alpha$ we see this is 
\begin{equation}\label{bound4}\ll \frac{e^{-\frac{\kappa^2 \log\log X}{2}}}{\alpha\sqrt{\log\log X}} e^{-\frac{\kappa\mathbf{s}}{2}\sigma^2_{r+1}},\end{equation}
for all sufficiently large $X$, and hence 

\begin{equation}\label{PnotA}
	\P(G_{r}\cap A_{r+1}^c)\ll   \frac{e^{-\frac{\kappa^2 \log\log X}{2}}}{\alpha\sqrt{\log\log X}} e^{-\frac{\kappa\mathbf{s}}{2}\sigma^2_{r+1}}, 
\end{equation}
as required.

We now turn to bound $ \P(H\cap G_r\cap B_{r+1}^c).$ 
The event $B_{r+1}^c$ concerns the random walk $S_{{r+1}}=\sum^{r+1}_{j=1}\mathscr{P}_{j}$, which is closely correlated to the mollifier $$M_{r+1}=\prod_{j\le r+1}\mathscr{A}_j.$$ We want to connect these quantities, to show that on the event $B_{r+1}^c$, $M_{r+1}$ is abnormally large. Since we expect the mollifier to be approximately inversely proportional to the central value, this should mean that large deviation event $H$ is rare, which we show using Markov's inequality. The following lemma connects  $\mathscr{P}_{{i}}$ to $\mathscr{A}_{i}$ for $0\le i \le R-1,$ which we combine to get bounds on $M_{r+1}.$
\newpage
\begin{lem}\label{Mollbig}
	Let $1\le i\le R, $ and suppose $d\in A_{i}\cap B_{i-1}.$ If $\mathscr{P}_{i}\le 0,$ we have 
	\begin{equation}\label{trivbound}
		\mathscr{A}_i\ge 1, 
	\end{equation}
while if $\mathscr{P}_{i}> 0,$ then
	\begin{equation}\label{inB}
	\mathscr{A}_{i}\ge	e^{-\mathscr{P}_{i}} \left(1+O\left(e^{-{l_{i}}}\right)\right).
	\end{equation} 
\end{lem}
\begin{proof}[Proof of Lemma \ref{Mollbig}]
	Equation \eqref{trivbound} is clear, since all the summands in Equation \eqref{defaj} are non-negative if $\mathscr{P}_{i}\le 0.$  We turn our attention to Equation \eqref{inB}, where $\mathscr{P}_i>0$. If the power series for $-\mathscr{P}_{i}$ weren't truncated to form $\mathscr{A}_{i}$, we would have $e^{-\mathscr{P}_{i}}$.  Rankin's trick yields the pointwise bound for any $\rho>0:$
	\begin{equation}\label{truncrank}
		\mathscr{A}_{i}=\exp(-\mathscr{P}_{i})+O\left(\exp\left(-\rho (l_{i-1}-l_i)^{10^5}\right)\sum^\infty_{j=0} \frac{e^\rho \mathscr{P}_{i}^j}{j!}\right)
		\end{equation}
		The sum may be expressed as $\exp\left(e^\rho \mathscr{P}_{i}\right)$. But since $d\in A_{i}\cap B_{i-1}$, we know the jump $\mathscr{P}_{i}$ in the interval $P_{i}$ cannot be large. Indeed, we  have \begin{align}\nonumber\mathscr{P}_{i}&= S_{i}-S_{{i-1}}\\&\le U_{i}-L_{i-1}\nonumber\\\nonumber&\le \kappa\left(n_{i+1}-n_{i}\right)+2\mathbf{s}\sigma^2_{i}\\&\le\kappa\left( \log l_{i}-\log l_{i+1}\right)+2\mathbf{s}\sigma^2_{i},\end{align}  
			 and hence 		
		\begin{equation}
			\exp\left(e^\rho \mathscr{P}_{i}\right)\le \exp((e^\rho+1)(\kappa\left( \log l_{i}-\log l_{i+1}\right)+2\mathbf{s}\sigma^2_{i}))\exp(-\mathscr{P}_{i}). 
		\end{equation} 
		Then Equation \eqref{truncrank} yields
		\begin{equation}\mathscr{A}_{i+1}=\exp(-\mathscr{P}_{i+1}) \left(1+O\left(\left(\exp((e^\rho+1)(\kappa\left( \log l_{i}-\log l_{i+1}\right)+2\mathbf{s}\sigma^2_{i})-\rho (l_{i}-l_{i+1})^{10^5})\right)\right)\right).\end{equation}
	We have $l_{i+1}-l_i>\mathbf{s}10^5$, and $l_{i+2}\le 2\mathbf{s}\log l_{i+1}+1$  so that upon putting $\rho=1000$, the error term is
	
	$O\left(\exp(3e^{1000}\mathbf{s}(2\mathbf{s}\log l_{i+1}+1)-500\mathbf{s}^2 l_{i}^{10^4})\right).$\\
	Our choice of $\mathbf{s}$ ensures
	\begin{equation}\label{bound2}
		3e^{1000}\mathbf{s}(2\mathbf{s}\log l_{i}+1)-500\mathbf{s}^2 l_{i}^{10^4}\le-l_i,
\end{equation}
so the error term is $O(\exp(-l_i)),$ as required. 
\end{proof}

In order to bound $\P(H\cap G_r\cap B_{r+1}^c),$ we partition on the values of $S_{r}$ and of the jump $\mathscr{P}_{r+1}=S_{{r+1}}-S_{r}.$ The event $G_r$ means $S_{r}\in [L_r,U_r]$, while $B_{r+1}^c$ ensures $S_{{r+1}}\le L_{r+1}.$ Hence we may decompose the probability of going below the barrier for the first time at $n_r$ as

\begin{equation}\label{Pvsplit}
	\P(H\cap G_r\cap B_{r+1}^c)\ll  \sum_{\substack{u\in [L_r,U_r]\\ u+v \le L_{r+1}}}\P\left(H\cap G_r\cap\{(S_{r}\in [u,u+1], S_{{r+1}}-S_{r}\in [v,v+1]\} )	
	\right).\end{equation} 
 We bound each probability using Markov's inequality, and to do so will have to factor in the large central value using the twisted mollifier formula, and bound all the relevant terms.
   A large value of $S_{{r+1}}-S_{r}$ is rare, and we bound its contribution using Markov's inequality on $|S_{{r+1}}-S_{r}|^{q_v},$ where \begin{equation}\label{eqn: qv}
 	q_v=\left\lceil \frac{v^2}{2\left(n_{r+1}-n_{r}\right)}\right\rceil.
 \end{equation} If $S_{{r+1}}-S_{r}$ is small, say less than $5$, we cannot show these values are rare. Instead, we can use Markov's inequality to show the event $H$ is very rare conditional on such small growth over the interval $P_r$. 
 In order to use Markov's inequality to show the central value cannot be too large very often, we observe that on $H$, we have \begin{equation}
 	L\ge \frac{e^V}{\sqrt{\log X}},
 \end{equation}
 while we use Lemma \ref{Mollbig} to bound the mollifier $M_{r+1}$. 
 We bound the contribution of the twists  with $ S_{{r+1}}-S_{r}\ge 5$ and those with $S_{{r+1}}-S_{r}\le 5$ separately. 
 If $ S_{{r+1}}-S_{r}\ge 5$, then we see that on $G_{r},$ Lemma \ref{Mollbig} yields
\begin{equation}
	M_{r+1}\gg \prod^{r+1}_{i=1}\exp(-\left(S_{i}-S_{{i-1}}\right))= \exp(-S_{r}-(S_{{r+1}}-S_{r})).
\end{equation}

The contribution to the probability $\P(H\cap G_r\cap B_{r+1}^c)$ of the characters with $S_{{r+1}}-S_{r}\ge 5$ may be bounded as

\begin{align}
	\ll \frac{\sqrt{\log X}}{e^V}\sum_{\substack{u\in [L_r,U_r]\\ u+v \le L_{r+1}\\ v\ge 5}}\E\left[L M_{r+1} \frac{|S_{{r+1}}-S_{r}|^{2q_v}}{v^{2q_v}}\mathbf{1}(S_{r}\in [u,u+1]\cap G_r)\right] e^{u+v} .
\end{align}

Since $5\le v\le L_{r+1}-L_r,$ we see Lemma \ref{walkmom} and Equation \eqref{2.7} from Lemma \ref{pointmom} apply with the choice of $q_v$ from Equation \eqref{eqn: qv}, and so the above is 
\begin{align}
	&\nonumber\ll \frac{\sqrt{\log X}}{e^V\sqrt{\log X_r}}\sum_{\substack{u\in [L_r,U_r]\\ u+v \le L_{r+1}\\ v\ge 5}}\frac{e^{-\frac{u^2}{\log \log X_r}}}{\sqrt{\log\log X}} \frac{(2q_v)!}{2^{q_v}q_v!v^{2q_v}}\left(n_{r+1}-n_{r}\right)^{q_v}e^{u+v}\\
	&\label{bigv}\ll  \frac{\sqrt{\log X}}{e^V\sqrt{\log X_r \log\log X}}\sum_{\substack{u\in [L_r,U_r]\\ u+v \le L_{r+1} \\ v\ge 5}}e^{-\frac{u^2}{\log \log X_r}+u}e^{v-\frac{v^2}{n_{r+1}-n_{r}}}.
\end{align}
It now remains to consider the contribution of the values with $S_{{r+1}}-S_{r}\le 5$ to Equation \eqref{Pvsplit}.

By Equations \eqref{trivbound} and \eqref{inB}, we see that if $v\le 5$ and  $d\in G_r$, then\begin{equation}A_{r+1}\ge e^{-5}\left(1+O\left(e^{-l_r}\right)\right).\end{equation}
Using Equation \eqref{trivbound} for $i\le r$ now yields that
\begin{equation}
	M_{r+1}\ge e^{-5}(1+O(e^{-l_r}))e^{-S_{r}}.
\end{equation}
Now that we have bounded below all the necessary quantities on $H\cap G_r\cap B_{r+1}^c$, we can use Markov's inequality to show \begin{align}\nonumber
	 &\sum_{\substack{u\in [L_r,U_r]\\ u+v \le L_{r+1}\\v\le 5}}\P\left(H\cap G_r\cap\{S_{r}\in [u,u+1], S_{{r+1}}-S_{r}\in [v,v+1]\}	
	\right)\\&\le \frac{\sqrt{\log X}}{e^V}\sum_{u\in [L_r,U_r]}\E\left[LM_{r+1}\mathbf{1}(S_{r}\in [u,u+1])\right]e^{u+5+O(e^{-l_r})}.
\end{align}
Using Equation \eqref{2.7} with $Q=1$ to evaluate the expectation, we see this is 
\begin{equation}
\label{smallv}\ll \log^{-\frac{1}{2}}(X_r) \sum_{u\in [L_r,U_r]}\frac{e^{-\frac{u^2}{n_r}}}{\sqrt{n_r}}e^{u+4}.
\end{equation}
By the construction of the intervals $X_j$, we see that if $v\le 5$ and $u\le U_r$, then $u+v\le L_{r+1}.$ Then combining Equations \eqref{bigv} and \eqref{smallv} yields 
\begin{equation}
	\P(H\cap G_r\cap B_{r+1}^c)\ll  \frac{\sqrt{\log X}}{e^V\sqrt{\log X_r\log\log X}}\sum_{\substack{u\in [L_r,U_r]\\ u+v \le L_{r+1} \\ v\ge 4}}e^{-\frac{u^2}{n_r}+u}e^{v-\frac{v^2}{n_{r+1}-n_{r}}}.
\end{equation}

We now perform the calculation for the Gaussian sum as in \cite{AB}.\\
We put $\tilde{u}=u-\kappa \log \log X_r$ where we subtract the midpoint of the interval $[L_r,U_r],$ so that $|\tilde{u}|\le \mathbf{s}l_{r+2},$ and $\tilde{v}=v-\kappa(n_{r+1}-n_{r}).$ Then, lifting the restriction $v\ge 4$, we see the above bound on the probability is

\begin{equation}
	\ll   \frac{\sqrt{\log X}}{e^V\sqrt{\log X_r\log\log X}}e^{-\kappa^2\log X_{r+1}}e^{\kappa n_{r+1}}\sum_{\substack{|\tilde{u}|\le \mathbf{s} l_{r+2},\\ \tilde{u}+\tilde{v} \le - \mathbf{s} l_{r+3}}}e^{(1-2\kappa)(\tilde{u}+\tilde{v})}e^{-\frac{\tilde{v}^2}{n_{r+1}-n_{r}}}.
\end{equation}
Performing first the sum over $\tilde{u}+\tilde{v}$, then $\tilde{v}$, we see that since $\kappa<\frac{1}{2}$ for $X$ sufficiently large, this is

\begin{align}
	&\nonumber\ll  \frac{\sqrt{\log X}}{(\frac{1}{2}-\kappa)e^V\sqrt{\log X_r\log\log X}} e^{-\kappa^2\log X_{r+1}-((1-2\kappa)\mathbf{s}+\kappa) l_{r+3}}\sqrt{n_{r+1}-n_{r}}\\
	&\ll \frac{e^{-\kappa^2\log X}}{\sqrt{\log\log X}} \frac{e^{\left(\frac{1}{2}+\kappa^2-\kappa-(1-2\kappa)\mathbf{s}\right)l_{r+3}}}{\frac{1}{2}-\kappa}.
\end{align}
Since $\mathbf{s}=\frac{10^5}{1-2\alpha},$ we see \begin{equation}\label{bound3}
	\frac{1}{2}+\kappa^2-\kappa-(1-2\kappa)\mathbf{s}<1-10^5
\end{equation}
if $X$ is sufficiently large, and so  \begin{equation}\label{PnotB}
\P(H\cap G_r\cap B_{r+1}^c)\ll \frac{e^{-\kappa^2\log X}}{\sqrt{\log\log X}} e^{-\delta l_{r+3}}\end{equation} for some $\delta>0$.
Substituting Equations \eqref{PnotA} and \eqref{PnotB} into Equation \eqref{decompmid} completes the proof of Proposition \ref{barrierprob} (i).
\subsection{Proof of Proposition \ref{barrierprob} (ii)}
It now remains to bound the probability of the walk staying within the barrier until the end. We partition on the value of $S_{R}$, to write
\begin{align}
	\P(G_{R})\ll \sum_{u\in [L_R,U_R]}\E\left[\mathbf{1}(S_R\in [u,u+1])\right].
\end{align}
Using Equation \eqref{2.4} with $Q=1$ shows this is

\begin{align}
	\ll \sum_{u\in [L_R,U_R]}\frac{e^{-\frac{u^2}{n_R}}}{\sqrt{n_R}}.
\end{align}
By the construction of $R$, we see $\log\log X-n_R=O\left(\frac{1}{\alpha}\right)+O\left(\frac{1}{1-2\alpha}\right).$ By comparing this to the law  of the random variable $Z\sim N(0,n_R),$ we see  \begin{align}
	\P(G_R)&\nonumber\ll \frac{e^{-\frac{L_R^2}{n_R}}}{\alpha\sqrt{\log \log X}}\\& \ll \frac{e^{-\frac{V^2-2V\left(O\left(\frac{1}{\alpha}\right)+O\left(\frac{1}{1-2\alpha}\right)\right)}{n_R}}}{\sqrt{\log\log X}}.\end{align}
	We recall $V=\alpha n_R+O(1),$ to show this is
	 $$\ll \frac{e^{-\frac{V^2}{\log \log X}}}{\sqrt{\log\log X}}\exp\left(O\left(\frac{1}{1-2\alpha}\right)\right).$$
	 This completes the proof of Proposition \ref{barrierprob} (ii).

\section{Proof of Proposition \ref{twistmollprop}}\label{SecTM}
\subsection{Proof method}
The proof of the bounds on the first moment of the twisted mollifier follow the methods in \cite{ABR20} for twisted mollified moments of the zeta function, and in \cite{AC25} for the $q$-aspect. We first quote Proposition 2 from \cite{RS2}, which gives an asymptotic for individual  twists of the form \begin{equation}\label{jointmoment}
	S(X;n,v)=\sum_{d\in\mathcal{E}(\mathscr{O},a,v)} L\left(\frac{1}{2},\chi_d(n)\right) \phi\left(\frac{\kappa d}{X}\right),\end{equation}
	with $(n,v)=(nv,N_0)=1$ and $v$ square-free.
We then substitute the expansion for the twisted mollifier $MQ^2$ and take the smoothed expectation of each term in the twist. In order to evaluate the twisted mollified moments, we use Rankin's trick to lift the effect of the smoothing, and then evaluate the entire sum. The calculations are simpler, at least notationally, than in \cite{AC25}, since we are only dealing with real characters, so don't have to handle the imaginary part of the twists. Moreover, we can neglect the contribution of squares of primes, since we are handling quadratic characters, which behave very predictably on squares.

In order to take twisted mollified moments, we must write the twisted mollifier in terms of its individual twists. Upon expansion, we see \begin{equation}
	M_rQ^2=\sum_n \frac{c_n \chi_d(n)}{n^{\frac{1}{2}}},
\end{equation} for some suitable choice of coefficients $\mathbf{c}=(c_1,\dots,c_N)$ of length $N=O(X^{\frac{1}{100}}).$ We require an approximate formula for the smoothed twisted moment:
\begin{equation}\label{deftwist}
	\mathscr{D}(\mathbf{c}):=\sum_{d\in\mathcal{E}(\mathscr{O},a,v)} L\left(\frac{1}{2}, E_d\right) \sum^N_{n=1}\frac{c_n \chi_d(n)}{n^{\frac{1}{2}}}\Phi\left(\frac{\kappa d}{X}\right).
\end{equation}

By taking the expectation of each twist $\chi_d(n)$, we see the above may be written as
	
	\begin{equation}\label{twist+prop2}
\sum^N_{n=1}\frac{c_{n}}{n^{\frac{1}{2}}} S(X;n,v) .
	\end{equation}
	We now recall Proposition 2 from \cite{RS2} to approximate the terms $S(X;n,v).$
	\begin{prop}[Proposition 2 in \cite{RS2}]\label{1twist}
		Let $S(X;n,v)$ be as defined in Equation \eqref{jointmoment}, and $n=mr^2$ with $m$ square-free. Then 
		\begin{equation}
			S(X;n,v)=
		 \frac{2Xa\left(m\right)}{v m^{\frac{1}{2}}N_0}\check{\Phi}(0)L_a\left(\frac{1}{2}\right)L(1, \text{sym}^2E)\mathcal{G}(1;n,v)+O(X^{\frac{7}{8}+\epsilon}n^{\frac{3}{8}}v^{\frac{1}{4}}),
			\end{equation}
			
			where we may write: \begin{equation}\label{defG}
			\mathcal{G}(1;n,v)=Cg(n)h(v),
			\end{equation}
			with $C=C(E)$ a non-zero constant and $g$ and $h$ multiplicative functions with $g(p^k)=1+O\left(\frac{1}{p}\right)$ and $h(p)=1+O\left(\frac{1}{p}\right).$
		
	\end{prop}

We use Proposition \ref{1twist} to handle each term arising in the twisted mollified moments. If we label the coefficients of the mollifier, so that
\begin{equation}\label{Mcoef}M_r=\sum_{t}\frac{e_t\chi_d(t)}{t^{\frac{1}{2}}},\end{equation} and the coefficients of the twist $Q$ are as defined in Equation \eqref{defQd}, then we see that 
\begin{equation}\label{coefprod}
	c_n=\sum_{tw_1w_2=n}e_tC(w_1)C(w_2).
\end{equation}
Upon substituting this into Equation \eqref{twist+prop2} and using Proposition \ref{1twist}, we see
\begin{align}
		\mathscr{D}(\mathbf{c})&=\sum_{n}\frac{\sum_{tw_1w_2=n}\nonumber e_{t}C(w_1)C(w_2)}{n^{\frac{1}{2}}}\times\\&\left( \frac{2Xa\left(m\right)}{v \left(m\right)^{\frac{1}{2}}N_0}\hat{\Phi}(0)L_a\left(\frac{1}{2}\right)L(1, \text{sym}^2E)\mathcal{G}(1;n,v)+O(X^{\frac{7}{8}+\epsilon}(n_1n_2)^{\frac{3}{8}}v^{\frac{1}{4}})\right).\label{firstD} \end{align}

We rewrite the above equation to group together terms with the same coefficients  $C(w_1)C(w_2)$ coming from the twist together. We have the following convenient relation between squarefree parts: if $n=tw_1w_2$, where $t=f^2x$, $w_1w_2=h^2y$, with $x$ and $y$ squarefree, then $m=\frac{xy}{(x,y)^2}.$ Moreover, $y$ only depends on the squarefree parts of $w_1$ and $w_2,$ and we may write \begin{equation}
	y=\prod_{p\le X_r } p^{\xi_p( w_1w_2)},\end{equation} where we define the parity \begin{equation}\xi_p(w):=2\left\{\frac{v_p(w)}{2}\right\}.\end{equation}
	This parity is $0$ if $p$ divides $w$ to an even power, and so does not divide the squarefree part of $w$, and $1$ if $p$ divides the squarefree part. Thus the diagonal terms correspond to the pairs $(w_1,w_2)$ where $\xi_p(w_1,w_2)=0$ for all $p\le X_R.$

Hence we may rewrite the expression for $\mathscr{D}(\mathbf{c})$ in Equation \eqref{firstD} as:
\begin{align}&\nonumber
		\sum_{w_1,w_2=1}\frac{C(w_1)C(w_2)}{(w_1w_2)^{\frac{1}{2}}}\left(\sum_{t} \frac{e_{t}}{t^{\frac{1}{2}}} \frac{2Xa\left(\frac{xy}{(x,y)^2}\right)}{v \left(\frac{xy}{(x,y)^2}\right)^{\frac{1}{2}}N_0}\hat{\Phi}(0)L_a\left(\frac{1}{2}\right)L(1, \text{sym}^2E)\mathcal{G}(1;tw_1w_2,v)\right)\\&
+O\left(X^{\frac{8}{9}}v^{\frac{1}{4}}\sum_{w}\frac{C(w)^2}{w} \max_t\frac{\left|e_{t}\right|}{t}
\right).
\end{align}
Here, we used the inequality \begin{equation}\label{CS2}
	\left|XY\right|\le \frac{1}{2}\left(X^2+Y^2\right)
\end{equation}to simplify the error term, and bounded the size of $w$ and $t$.

Since the Dirichlet polynomials $M_r$ and $Q$ split into factors determined by primes in the intervals $P_j,$ and $a(n)$ and $\mathcal{G}(1;t,v)$ are multiplicative, we see

\begin{align}\label{wellfactorbound}
	\mathscr{D}(\mathbf{c})&\ll\frac{X\hat{\Phi}(0)L_a\left(\frac{1}{2}\right)L(1,\text{sym}^2E)}{vN_0}\prod^{r}_{j=1}\mathscr{N}_j\\\nonumber&+O\left(X^{\frac{8}{9}}v^{\frac{1}{4}}\prod^r_{j=1}\left(\max_{p|t\implies p\in{(X_{j+1},X_j]}}\left( \frac{\left|e_t\right|}{t}\right)\sum_{p|w\implies p\in (X_{j+1},X_{j}]}\frac{\gamma(w)^2}{w}\right)\right),
\end{align}
where \begin{equation}\label{deftilN}
	\mathscr{N}_j=\sum_{\substack{p|q_1q_2\implies p\in P_{j}\\ q_1,q_2\text{ square-free}}}\left|\sum_{\substack{p|w_1w_2\implies p\in P_{j}\\ w_1q_1,w_2q_2=\square}}\frac{\gamma(w_1)\gamma(w_2)}{(w_1w_2)^{\frac{1}{2}}}\sum_{\substack{p|t\implies p\in (X_{j+1},X_{j}]}}			
	\frac{e_{t}}{t^{\frac{1}{2}}} \frac{2Xa\left(\frac{xy}{(x,y)^2}\right)}{v \left(\frac{xy}{(x,y)^2}\right)^{\frac{1}{2}}N_0}\mathcal{G}(1;tw_1w_2,v)
	\right|
\end{equation}
for $1\le j\le r$.

In Lemma \ref{Chernprop}, we will use Rankin's trick to remove the restriction on the truncation of the mollifiers $\exp(-\mathscr{P}_j)$ at $\mathscr{A}_j,$ replacing $\mathscr{N}_j$ with $\tilde{\mathscr{N}}_j.$ In Lemma \ref{twistmoll}, we will then evaluate the sum without the restriction, completing the proof of Proposition \ref{twistmollprop}.

We define coefficients $\tilde{e}_t$ given by the Dirichlet series relation:
\begin{equation}\label{expDirich}
	\exp\left(-\sum_{p\in P_j} \frac{a(p)\chi_d(p)}{p^{s}}\right)=\sum_{p|t\implies p\in P_j} \frac{\tilde{e}_t \chi_d(t)}{t^s}.
\end{equation}
These would be the coefficients of the mollifier $\exp\left(-\mathscr{P}_j\right)$ without the truncation used to constructional $\mathscr{A}_j$. In Lemma \ref{Chernprop}, we show the effect of the additional terms after the truncation are negligible.

We will see in Lemma \ref{twistmoll}, that if we could discard the restriction on the prime factorisations, then the coefficient multiplying each twist $\frac{\gamma(w_1)\gamma(w_2)}{(w_1w_2)^{\frac{1}{2}}} $ in Equation \eqref{deftilN} would factor as an Euler product $\prod_{p\in P_{j}} 	\beta_p(w_1w_2)$, 	where the Euler factor is defined as: 
\begin{equation}\label{betaf}
	\beta_p(w)=\sum^{\infty}_{i=0}  \frac{a(p)^{2i+\xi_{p}(w)}}{p^{i+\frac{\xi_p(w)}{2}}(2i)!}\mathcal{G}_p(1;p^{2i+v_p(w)},v)-\sum^\infty_{i=1}\frac{a(p)^{2i-\xi_{p}(w)}}{p^{i-\frac{\xi_p(w)}{2}}(2i-1)!} \mathcal{G}_p(1;p^{2i+v_p(w)-1},v).
\end{equation}
We use Rankin's trick to bound the effect off the additional terms in the mollifier by a small multiple of the sum. Then when we evaluate the main term in Lemma \ref{twistmoll} with the restriction lifted, we will see the off-diagonal terms become negligible.

	\begin{lem}\label{Chernprop}
	With the notations as above, we have
	
	\begin{equation}\begin{split}\label{wellfactorsmooth}
			\mathscr{D}(\mathbf{c})\ll \frac{X\hat{\Phi}(0)L_a\left(\frac{1}{2}\right)L(1,\text{sym}^2E)h(v)}{vN_0}\prod^{r}_{j=1}\tilde{\mathscr{N}}_j+O\left(X^{\frac{9}{10}}v^{\frac{1}{4}}\max_w\frac{C(w)^2}{w}\right),
	\end{split}\end{equation} where
	\begin{align}\label{defN}\tilde{\mathscr{N}}_j=\sum_{\substack{p|q_1q_2\implies p\in P_{j}\\ q_1,q_2\text{square-free}}}\left|\sum_{\substack{p|w_1w_2\implies p\in P_{j}\\ w_1q_1,w_2q_2=\square}}\frac{\gamma(w_1)\gamma(w_2)}{(w_1w_2)^{\frac{1}{2}}}\times
		\sum_{\substack{p|t\implies p\in (X_{j+1},X_{j}]}}			
		\frac{\tilde{e}_{t}}{t^{\frac{1}{2}}} \frac{a\left(\frac{xy}{(x,y)^2}\right)}{ \left(\frac{xy}{(x,y)^2}\right)^{\frac{1}{2}}}g(tw_1w_2)
		\right|\\+O\left(\exp(-99(l_{j+1}-l_{j})^2)\sum_{\substack{p|q_1\implies p\in P_{j}\\ q_1\text{  squarefree}}}\sum_{q_1w,q_1w=\square} \frac{\gamma(w_1)\gamma(w_2)}{(w_1w_2)^{\frac{1}{2}}} \prod_{p\in P_{j}} 	\beta_p(w_1w_2)\right).
	\end{align}
\end{lem}

Here, as before, we write $t=f^2x$ and $w_1w_2=h^2y$, with $x$ and $y$ squarefree, and note that the restriction on the prime factors of $t$ and $w$ lying in the interval $P_j$ extends to their factors.
 
\begin{rem}
	
	The key observation which complicates the proof of Proposition \ref{Chernprop} comes from the fact that the diagonal terms include all those with square products, not just the twists which are equal. 
In general, there may be cancellation among the diagonal terms with square product, and this cancellation must be factored into adding any extra terms in the mollifier. This means more care must be taken   compared to Lemma 3.6 in \cite{AC25}. \\ For example, if we took a twist with $Q_d(\frac{1}{2})=\chi_d(p^2)-\chi_d(q^2)$ for large distinct primes $p,q\approxeq X_r,$ then  $Q_d(\frac{1}{2})$ would nearly always be $0$. However, the individual error terms from the extra coefficients from the mollifier multiplying the cross terms $\chi_d(p^4), \chi_d(p^2q^2)$ and $\chi_d(q^2)$ may be large, and so if taken individually would become dominant.
	
	Unlike in \cite{AC25}, where the extra terms in the mollifier could be considered separately at each twist, in order to preserve the cancellation in diagonal terms, we must consider each extra term in the mollifier and the combined contribution of all the twists.
\end{rem}

With the truncation lifted, we can then evaluate the bound for the twisted mollified moment.  

\begin{lem}\label{twistmoll}
	Using the above notation, we have,
	\begin{align}\label{wellfactorconj}
		\prod^{r}_{j=1}\tilde{\mathscr{N}}_j&\ll\log ^{-\frac{1}{2}}(X_r)\sum_{\substack{p|q\implies p\le X_r\\ q\text{ square-free}}}\sum_{\substack{p|u_1u_2\implies p\le X_r\\ u_1q,u_2q=\square}} \frac{C(u_1)C(u_2)}{(u_1u_2)^{\frac{1}{2}}}\prod_{p|u_1u_2}\left(1-\frac{1}{p
		}\right).
	\end{align}
	Moreover, each square-free $q$ has a non-negative summand in the above equation. 
\end{lem}
 
Substituting the bound from Lemma \ref{twistmoll} into Lemma \ref{Chernprop} will then complete the proof of Proposition \ref{twistmollprop}.
\subsection{Proofs of Lemmata \ref{Chernprop} and \ref{twistmoll}} In order to prove Lemma \ref{Chernprop}, we show we can add the contribution of the terms with many prime factors to Equation \eqref{defN} which were not present in Equation \eqref{deftilN} with a small error. In order to bound the contribution of the extra terms in the mollifier, we appeal to Lemma \ref{switchcher}, which simplifies the use of Rankin's trick. We then prove Lemma \ref{twistmoll}, where we see once the effects of the truncation has been lifted, the diagonal twists with square product dominate the expectation. In order to bound the off-diagonal terms, we use Lemma \ref{freebound} to show that the total contribution of all off-diagonal twists is negligible. The proofs of Lemmata \ref{switchcher} and \ref{freebound} are deferred to the end of the section.
\begin{proof}[Proof of Lemma \ref{Chernprop}.]
		
		In order to get a short enough Dirichlet polynomial $\mathscr{A}_j$ to apply Proposition \ref{1twist}, we had to truncate the expansion of $\exp\left(-\mathscr{P}_j\right)$ at some large power $(l_j-l_{j+1})^{10^5}.$ However, to express the resulting expression for the coefficients in the twisted mollifier formula as an Euler product, we need the full sum.  
		We thus want to use Rankin's trick to show we can lift the restriction on the truncation at $l_j$ of powers of $-\mathscr{P}_j.$
		
		We recall the desired mollifier from Equation \eqref{expDirich}. We would like to approximate $\mathscr{N}_{j}$ by the multiplicand  
		\begin{align}\label{hatN}\hat{\mathscr{N}}_j=\sum_{\substack{p|q_1q_2\implies p\in P_{j}\\ q_1,q_2\text{ square-free}}}\left|\sum_{\substack{p|w_1w_2\implies p\in P_{j}\\ w_1q_1,w_2q_2=\square}}\frac{\gamma(w_1)\gamma(w_2)}{(w_1w_2)^{\frac{1}{2}}}
			\sum_{\substack{p|u\implies p\in P_j}}			
			\frac{\tilde{e}_{u}}{u^{\frac{1}{2}}} \frac{2Xa\left(\frac{xy}{(x,y)^2}\right)}{ \left(\frac{xy}{(x,y)^2}\right)^{\frac{1}{2}}}g(uw_1w_2)\right|.\end{align}
		We incur some error in this approximation, which gives the error terms in $\tilde{\mathscr{N}}_{j}.$		
		This approximation follows the use of Rankin's trick in proving Proposition 5 in \cite{RS2}, and we borrow much of the same notation.
		
		If we ignored the truncation of the exponential at some high power of $\mathscr{P}_j,$ we would have the main term $\hat{\mathscr{N}}_{j},$ with an Euler product at the prime $p$ for each product on the twist depending on the parity $\xi_{p}(w_1w_2)$.
		Indeed, we may express $\hat{\mathscr{N}}_{j}$ as:  
		
		\begin{equation}\label{newN}\begin{split}
				\sum_{\substack{p|q_1q_2\implies p\in P_{j}\\ q_1,q_2\text{ square-free}}}\left|\sum_{\substack{p|w_1w_2\implies p\in P_{j}\\ w_1q_1,w_2q_2=\square}}\frac{\gamma(w_1)\gamma(w_2)}{(w_1w_2)^{\frac{1}{2}}}\prod_{p\in P_j}\beta_p(w_1w_2) \right|,	\end{split}
		\end{equation}
		where $\beta_p(w_1w_2)$ is as defined in Equation \eqref{betaf}.
	We now use Rankin's trick to bound the contribution of the additional primes. 
	
	The contributions of the terms with $i\ge (l_j-l_{j+1})^{10^5}$ may be bounded by Rankin's trick for any $\rho>0$ 
	as

	\begin{equation}\label{Rankin1}\begin{split}
			\le e^{- \rho(l_j-l_{j+1})^{10^5}} \sum_{p|t\implies p\in P_j}\sum_{\substack{p|q_1q_2\implies p\in P_{j}\\ q_1,q_2\text{  squarefree}}} \left|\frac{e^{\rho \Omega(t)} e_ta\left(\frac{xy}{(x,y)^2}\right)}{t^{\frac{1}{2}}\left(\frac{q_1q_2xy}{(x,y)^2}\right)^{\frac{1}{2}}} \Delta(q_1,q_2,t)\left[\mbox{\boldmath$\gamma$}\right]\right|, 
		\end{split}
	\end{equation}
	where \begin{equation}\Delta(q_1,q_2,t)\left[\mbox{\boldmath$\gamma$}\right]=  \sum_{\substack{p|w_1w_2\implies p\in P_{j}\\ w_1q_1,w_2q_2=\square}}\frac{\gamma(w_1)\gamma(w_2)}{(w_1w_2)^{\frac{1}{2}}} \mathcal{G}(1;tw_1w_2,v).
	\end{equation}
	It is important to note that in this form, we are ensuring that any additional terms are multiplied by something as small as the expected value of the twist, since if two terms cancel in the twist, this cancellation will also be reflected in the sums $\Delta(q_1,q_2,t)\left[\mbox{\boldmath$\gamma$}\right].$

	We then bound $\Delta(q_1,q_2,t)\left[\mbox{\boldmath$\gamma$}\right]$  it in terms of the diagonal terms at $q_1$ and $q_2$.
	Note that $\Delta(q_1,q_2,t)$ defines a quadratic form on the variables $\gamma(w)$, where $w$ has square-free part $q_1$ or $q_2$. Moreover, the coefficients factor into a product of primes, since $\mathcal{G}$ does.
	We can thus use Lemma \ref{splitbound} to convert bounds on the quadratic forms restricted to powers of a single prime $p$ into bounds on the general quadratic form $\Delta(q_1,q_2,t)$. The factor at a single prime $p$ depends on whether $p$ divides the square-free parts of $q_1$ and $q_2$.
	Indeed, we quote the following Lemma, enabling us to bound the sum in Equation \eqref{Rankin1}.
			\begin{lem} \label{switchcher} For any real choice of coefficients $\mbox{\boldmath$\gamma$}$, and natural numbers $q_1,q_2$ and $t$, we have
		\begin{itemize}
			\item\label{posdeftarget} $\Delta(q_1,q_1,1)\left[\mbox{\boldmath$\gamma$}\right]\ge 0.$\\
			\item \label{bound}$\left|\Delta(q_1,q_2,t)\left[\mbox{\boldmath$\gamma$}\right]\right|\le \frac{1}{2} \left(\Delta(q_1,q_1,1)\left[\mbox{\boldmath$\gamma$}\right]+\Delta(q_2,q_2,1)\left[\mbox{\boldmath$\gamma$}\right]\right). $
		\end{itemize}
		\end{lem}
		\begin{rem}
			The summands defining $\Delta(q_1,q_1,1)$ all have the same square-free part $q_1$, and so essentially have the same twist at any modulus, subject to divisibility conditions. Hence the products of the twists should all be non-negative. In contrast, if $q_1\ne q_2,$ then $\chi_d(q_1)\chi_d(q_2)$ should be positive and negative with roughly equal proportions, so we expect much greater cancellation in $\Delta(q_1,q_2,t).$		It is important to observe that we have preserved any diagonal cancellation in Equation \eqref{bound}, by preserving the sums $\sum_m  \alpha_{p,d_1+2m}$ and $\sum_n \phi_{p,d_2+2n}$ in the form for the diagonal bounds $\Delta(q_1,q_1,1)$ and $\Delta(q_2,q_2,1).$
		\end{rem}
		We prove Lemma \ref{switchcher} in the next subsection and return to the proof of Proposition \ref{Chernprop}.
			We substitute the bounds on the mixed terms from Lemma \ref{switchcher} into Equation \eqref{Rankin1}.
		Hence, the contribution of the additional terms in the mollifier may be bounded as
		\begin{equation}\label{Rankin2}\begin{split}
				\le e^{-\rho(l_j-l_{j+1})^{10^5}} \sum_{\substack{p|q_1\implies p\in P_{j}\\ q_1\text{  squarefree}}}\Delta(q_1,q_1,1)  \sum_{\substack{p|q_2\implies p\in P_{j}\\ q_2\text{  squarefree}}}\frac{1}{(q_1q_2)^{\frac{1}{2}}}\sum_{p|t\implies p\in P_j}\frac{e^{\rho\Omega(t)} e_t}{t^{\frac{1}{2}}}
				\left|\frac{a\left(\frac{xy}{(x,y)^2}\right)}{\left(\frac{xy}{(x,y)^2 }\right)^{\frac{1}{2}}}\right|. 
			\end{split}
		\end{equation}
		We may  write the sum over $t$ as an Euler product as in the proof of Proposition 5 in \cite{RS2}, to express it as:		
		\begin{equation}\prod_{p\in P_{j}}\phi_p(y),\end{equation} where
		\begin{equation}
			\phi_p(y)=\begin{cases}
				\sum^\infty_{j=0} \frac{a(p)^{2j}}{p^j} \frac{e^{2\rho j}}{(2j)!}+\sum^{\infty}_{j=0} \frac{a(p)^{2j+2}}{p^{j+1}}\frac{e^{\rho(2j+1)}}{(2j+1)!}&p\nmid y\\
				\sum^\infty_{j=0} \frac{a(p)^{2j+1}}{p^{j+\frac{1}{2}}} \frac{e^{2\rho j}}{(2j)!}+\sum^{\infty}_{j=0} \frac{a(p)^{2j+1}}{p^{j+\frac{1}{2}}}\frac{e^{\rho(2j+1)}}{(2j+1)!}&p|y\\			
			\end{cases}	.	\end{equation}
We may easily handle the product following the same method as in \cite{RS2} to show it is uniformly bounded by:

\begin{equation}
	\exp\left(e^{2\rho}\sum_{p\in P_{j}} \frac{a(p)^2}{p}\right)\frac{e^{\rho\Omega(y)}2^{\Omega(y)}}{y^{\frac{1}{2}}}.
\end{equation}

By the restriction on the number of prime factors of $q_1$ and $q_2$, this is 

\begin{equation}\label{psibound}
	\le 		\exp\left(e^{2\rho}\sum_{p\in P_{j}} \frac{a(p)^2}{p}\right)\frac{e^{40\rho(l_j-l_{j+1})^{10^4}}}{y^{\frac{1}{2}}}.
\end{equation}		
		Substituting this bound for the sum over $t$ into Equation \eqref{Rankin2}, we see the internal sum over $q_2$ may now be bounded as
\begin{equation}\le \exp\left(e^{2\rho}\sum_{p\in P_{j}} \frac{a(p)^2}{p}+40\rho(l_j-l_{j+1})^{10^4}\right)\prod_{p\in P_{j}} \begin{cases}
		1+\frac{2}{p}+O\left(\frac{1}{p^2}\right)& p\nmid q_1\\
		\frac{2}{p}+O\left(\frac{1}{p^2}\right)& p\mid q_1
	\end{cases}.\end{equation}
Since $q_1$ is square-free and has at most $10(l_j-l_{j+1})^{10^4}$ prime factors, the above is bounded by:
\begin{equation}
	\frac{1}{q_1} \exp\left(e^{2\rho}\sum_{p\in P_{j}} \frac{a(p)^2}{p}+(40\rho+10)(l_j-l_{j+1})^{10^4}\right)\prod_{p\in P_{j}}1+\frac{2}{p}+O\left(\frac{1}{p^2}\right). 
\end{equation}	
Simple applications of the Prime Number Theorem to handle the product and of Lemma 3 in \cite{RS2} to handle the term $\sum_{p\in P_{j}} \frac{a(p)^2}{p},$ now shows this may be bounded by:

\begin{equation}
	\frac{1}{q_1} \exp\left((40\rho+10)(l_j-l_{j+1})^{10^4}+(e^{2\rho}+3)(n_j-n_{j-1}+O(1))\right). 
\end{equation}

We substitute this bound for the sum over $q_1$ into Equation \eqref{Rankin2} to show the total contribution of the extra terms with too many prime factors to appear in the mollifier may be bounded by:

\begin{align}&\exp\left(-\rho(l_j-l_{j+1})^{10^5}+(40\rho+10)(l_j-l_{j+1})^{10^4}+(e^{2\rho}+3)(n_j-n_{j-1}+O(1))\right)\times\nonumber\\ &\label{casesplitchern}\sum_{\substack{p|q_1\implies p\in P_{j}\\ q_1\text{  squarefree}}}\frac{\Delta(q_1,q_1,1)}{q_1}.\end{align}
If $j>1,$ then \begin{equation}
n_j-n_{j-1}=\log l_{j-1}-\log l_j\le l_{j}-l_{j+1}.\end{equation} For the last inequality, we must further subdivide into the cases $j=2$ and $j>2$.
 Taking $\rho=1000$, and recalling that we truncate such that $l_j-l_{j+1}\ge 10^4$ for all $j\le R$, we see the above is

\begin{equation}\le \exp\left(-99(l_{j+1}-l_{j})^2\right)\sum_{\substack{p|q_1\implies p\in P_{j}\\ q_1\text{  squarefree}}}\sum_{q_1w_1,q_1w_2=\square} \frac{\gamma(w_1)\gamma(w_2)}{(w_1w_2)^{\frac{1}{2}}} \prod_{p\in P_{j}} \beta_p(w_1w_2),\end{equation} as required. 

It remains to consider the case $j=1$. We proceed from Equation \eqref{casesplitchern}, and using that $n_1-n_0=\log\log X-2\log_3 X+O(1)$, whilst $l_0-l_1= 200\log\log X.$ Again  taking $\rho=1000$, we see the contribution of the extra terms is

\begin{equation}
	\ll \exp\left(-(\log\log X)^{10^5}\right)\sum_{\substack{p|q_1\implies p\in P_{1}\\ q_1\text{  squarefree}}}\sum_{q_1w_1,q_1w_2=\square} \frac{\gamma(w_1)\gamma(w_2)}{(w_1w_2)^{\frac{1}{2}}} \prod_{p\in P_{1}} \beta_p(w_1w_2),
\end{equation}  
which is completely negligible for large $X$. This completes the proof of Lemma \ref{Chernprop}.\end{proof}
Having proven Lemma \ref{Chernprop}, we can now return to the proof of Lemma \ref{twistmoll}, to bound the unrestricted sum.
\begin{proof}[Proof of Lemma \ref{twistmoll}] Now that we have lifted the restriction on the number of prime factors of the support of the sum of the coefficients, it remains to give an expression for $\tilde{\mathscr{N}}_{j}$.

	Returning to Equation \eqref{newN} and using Lemma \ref{Chernprop} to remove the truncation on the coefficients of the mollifier, we may write \begin{align}
			\tilde{\mathscr{N}}_j&=\sum_{\substack{p|q_1q_2\implies p\in P_{j}\\ q_1,q_2\text{ square-free}}}\left|\sum_{\substack{p|w_1w_2\implies p\in P_{j}\\ w_1q_1,w_2q_2=\square}}
			\frac{\gamma(w_1)\gamma(w_2)}{(w_1w_2)^{\frac{1}{2}}} \prod_{p\in P_{j}} 	\beta_p(w_1w_2)\right|+\\&\nonumber O\left(\exp(-99(l_{j+1}-l_{j})^2)\sum_{\substack{p|q_1\implies p\in P_{j}\\ q_1\text{  squarefree}}}\sum_{q_1w_1,q_1w_2=\square} \frac{\gamma(w_1)\gamma(w_2)}{(w_1w_2)^{\frac{1}{2}}} \prod_{p\in P_{j}} 	\beta_p(w_1w_2)\right).		\end{align}		

	The main term is bounded below by the diagonal terms with $q_1=q_2,$ which we recognise from the error term. Hence, we may make the error multiplicative to write 
	\begin{equation}\label{Multerrorn}	\tilde{\mathscr{N}}_j=\left(1+O\left(\exp(-99(l_{j+1}-l_{j})^2)\right)\right)\sum_{\substack{p|q_1q_2\implies p\in P_{j}\\ q_1,q_2\text{square-free}}}\left|\sum_{\substack{p|w_1w_2\implies p\in P_{j}\\ w_1q_1,w_2q_2=\square}}
			\frac{\gamma(w_1)\gamma(w_2)}{(w_1w_2)^{\frac{1}{2}}} \prod_{p\in P_{j}} 	\beta_p(w_1w_2)\right|. 			\end{equation}
		The sum consists of the diagonal terms where $q_1=q_2$ and the off-diagonal terms  where  $q_1\ne q_2$.
		We use Lemma \ref{splitbound} to bound the off-diagonal terms in term of the diagonal contribution, and hence show that the diagonal terms are dominant.
		
		We follow the proof of Proposition \ref{Chernprop}, and use much of the same notation.
		In order to show the off-diagonal forms are dominated by the diagonal terms, we show that for every additional prime for which $v_p(q_1)\ne v_p(w_2)$ mod $2$, the contribution of the cross term essentially gets smaller by a factor of $\frac{\beta_p(p)}{\beta_p(p^2)}=O(p^{-\frac{3}{2}}).$ This indeed means the diagonal terms with square product are dominant.
		We make this precise in the following Lemma. The proof is deferred to the next subsection.
		\begin{lem}\label{freebound}
			Let $q_1$ and $q_2$ be square-free natural numbers with prime factors contained in $P_{j}.$ Then \begin{align}
					\nonumber&	\left|\sum_{\substack{p|w_1w_2\implies p\in P_{j}\\ w_1q_1,w_2q_2=\square}}
					\frac{\gamma(w_1)\gamma(w_2)}{(w_1w_2)^{\frac{1}{2}}} \prod_{p\in P_{j}} 	\beta_p(p^{\xi_p(q_1q_2)})\right|\le\frac{1}{2}\prod_{p\in P_j} \exp\left(O\left(p^{-\frac{3}{2}}\right)-\frac{a(p)^2}{2p}\right) \prod_{\substack{r\in P_j\\ \xi_r(q_1q_2)=1}}s_r\\& \left(\sum_{\substack{p|w_1w_1'\implies p\in P_{j}\\ w_1q_1,w_1'q_1=\square}}
					\frac{\gamma(w_1)\gamma(w_1')}{(w_1w_1')^{\frac{1}{2}}} \prod_{p|w_1w_1'} 	\left(1-\frac{1}{p}\right)+\sum_{\substack{p|w_2w_2'\implies p\in P_{j}\\ w_2q_2,w_2'q_2=\square}}
					\frac{\gamma(w_2)\gamma(w_2')}{(w_2w_2')^{\frac{1}{2}}} \prod_{p|w_2w_2'} 	\left(1-\frac{1}{p}\right)\right),
			\label{noabs}		
			\end{align}	where $s_r=O\left(r^{-\frac{3}{2}}\right).$
		\end{lem} 
		\begin{rem} We will use that the sums on the right-hand side of Equation \eqref{noabs} are non-negative for any choice of $q_1$ and $q_2$.
			Moreover, the error term $\prod_{ \xi_r(q_1q_2)=1} 
		s_r$ is $1$ if $q_1=q_2$ and effectively means the contribution of the off-diagonal terms with distinct square-free parts is negligible. \end{rem}
		We now are in a position to conclude the proof of Lemma \ref{twistmoll}.	We return to Equation \eqref{Multerrorn} and use Lemma \ref{freebound} to bound the sum. 			
		Performing the sum over $q_2$ we may rewrite the right-hand side of Equation \eqref{Multerrorn} as   

	\begin{equation}
		\exp\left(O(l_{j+1}^{-\frac{1}{2}}-l_{j}^{-\frac{1}{2}})\right)\prod_{p\in P_j}\exp\left(-\frac{a(p)^2}{2p}\right)\sum_{\substack{p|q_1\implies p\in P_{j}\\ q_1\text{ square-free}}}\sum_{\substack{p|w_1w_1'\implies p\in P_{j}\\ w_1q_1,w_1'q_1=\square}}
		\frac{\gamma(w_1)\gamma(w_1')}{(w_1w_1')^{\frac{1}{2}}} \prod_{p|w_1w_1'} 	\left(1-\frac{1}{p}\right).
	\end{equation}
	Hence we obtain \begin{equation}
		\prod^r_{j=1}\tilde{\mathscr{N}}_{j}\ll \prod_{p\le X_r} \exp\left(-\frac{a(p)^2}{2p}\right)\sum_{\substack{p|q\implies p\le X_r\\ q\text{ square-free}}}\sum_{\substack{p|u_1u_2\implies p\le X_r\\ u_1q,u_2q=\square}} \frac{C(u_1)C(u_2)}{(u_1u_2)^{\frac{1}{2}}}\prod_{p|u_1u_2}\left(1-\frac{1}{p
		}\right).
	\end{equation}
	We may use Lemma 3 from \cite{RS2} to handle the product over primes $p\le X_r$, and see \begin{equation}\prod_{p\le X_r} \exp\left(-\frac{a(p)^2}{2p}\right)\ll \log^{-\frac{1}{2}}X_r.\end{equation} This completes the proof of Lemma \ref{twistmoll}, and so we may conclude the proof of Proposition \ref{twistmollprop}.

	\end{proof}
	\subsection{Proofs of Lemmata \ref{switchcher} and  \ref{freebound}}\label{Lemsec}
	In this subsection, we prove the necessary bounds on the off-diagonal terms with respect to the diagonal terms used for the proofs of Lemmata \ref{Chernprop} and \ref{twistmoll}. We view the diagonal and off-diagonal terms as quadratic forms on the coefficients, and use simple results on the theory from quadratic forms proven in Appendix \ref{Sec: QF}.
	\begin{proof}[Proof of Lemma \ref{switchcher}]
		We need to show the diagonal terms $\Delta(q_1,q_1,1)$ are non-negative definite quadratic form, and dominate the off-diagonal terms $\Delta(q_1,q_2,u).$ We utilise  Lemma \ref{splitbound} to simplify the proof.
		Here, the individual vector spaces are $\{V_p: p\in P_j\},$ where $$V_p=\text{span}\left\{\frac{\chi_d(p^n)}{p^{\frac{n}{2}}}: n\ge 0\right\}.$$ The dominant form $Z_p$ on $V_p$
		is given by the restriction of the sum: \begin{equation}\sum_{\substack{p|w_1w_2\implies p\in P_{j}\\ w_1r_1,w_2r_2=\square}}\frac{\gamma(w_1)\gamma(w_2)}{(w_1w_2)^{\frac{1}{2}}} \mathcal{G}(1;w_1w_2,v)\end{equation}
		to the case where $r_1,r_2\in\{1,p\},$ while the form $R_p$ is given by the restriction of: \begin{equation}\sum_{\substack{p|w_1w_2\implies p\in P_{j}\\ w_1r_1,w_2r_2=\square}}\frac{\gamma(w_1)\gamma(w_2)}{(w_1w_2)^{\frac{1}{2}}} \mathcal{G}(1;tw_1w_2,v).\end{equation} 
		
		Finally, $\alpha_{p,i}$ is supported  on prime powers with square-free part $p^{v_p(q_1)}$ and $\phi_{p,i}$ is supported  on prime powers with square-free part $p^{v_p(q_2)}.$ 
		
		For any choice of $t$, the coefficients of the quadratic form $\Delta(q_1,q_2,t)$  split  as an Euler product depending on the prime factorisations of $q_1$ and $q_2$. Hence, it suffices to show the positive definite and boundedness condition restricted to twists at powers of any fixed prime $p\in P_{j}$.
		
		That is, we put $\tau=v_p(q_1q_2),\xi_0=2\{\frac{\tau}{2}\}$ and $ d_j=2\left\{\frac{v_p(q_j)}{2}\right\}$ for $j=1,2$. Then if   $(\alpha_{p,i})$ and $(\phi_{p,i})$ are any choice of coefficients for the twists at the powers of $p$ with support on parities matching  $d_1$ and $d_2$ respectively, firstly we need to show that the form $\Delta(q_1,q_1,1)$ is positive definite, i.e. \begin{equation}\label{posdef}
			\sum^{\infty}_{m,n=0}\alpha_{p,d_1+2m}\alpha_{p,d_1+2n}  \mathcal{G}_p(1;p^{2m+2n},v)\ge 0.\end{equation}
		
		Secondly, to apply Lemma \ref{splitbound}, we need Equation \eqref{onedom}  to be satisfied. That is, we need to show the dominance of the forms:
		
		\begin{align} \label{mix}
			&\nonumber\left|\sum^{\infty}_{m,n=0}\alpha_{p,d_1+2m}\phi_{p,d_2+2n}  \mathcal{G}_p(1;up^{d_1+d_2+2m+2n},v)
			\right|\le\\& \frac{1}{2} \left(\left|\sum^{\infty}_{m,n=0}\alpha_{p,d_1+2m}\alpha_{p,d_1+2n}  \mathcal{G}_p(1;p^{2m+2n},v)
			\right|+\left|\sum^{\infty}_{m,n=0}\phi_{p,d_2+2m}\phi_{p,d_2+2n}  \mathcal{G}_p(1;p^{2m+2n},v)
			\right|\right).
		\end{align}
		Here, we rescaled the coefficients of the twist to cancel the $(w_1w_2)^{\frac{1}{2}} 
		$ factor for simplicity. 
		We begin with the proof of Equation \eqref{posdef}. Taking the product of Equation \eqref{mix} over all $p\in P_j$ and using Equation \eqref{mixdom} from Lemma \ref{splitbound} will then complete the proof of Lemma \ref{switchcher}.
		
		From the expressions for $\mathcal{G}_p(1;p^n,v)$ in Section 10 of \cite{RS2}, we see that there exists  values $\theta_1\ge \theta_2>0$ and $\theta_3\ge\theta_4>0 $ with $\theta_1-\theta_2\ge \theta_3-\theta_4,$ $\theta_2 \ge \theta_4$, such that
		\begin{equation}\label{pure}\mathcal{G}_p(1; p^{2j};v)=\begin{cases} \theta_1& j=0\\
				\theta_2 &j>0\end{cases},
		\end{equation}
		while
		\begin{equation}\label{twistfactor}\mathcal{G}_p(1; up^{2j+d_1+d_2};v)=\begin{cases} \theta_3& j=0\\
				\theta_4 &j>0\end{cases}.
		\end{equation}
		This essentially reflects the observation that if $d$ is any twist, then $\chi_d(p^{2n}r)=\chi_d(r),$  whenever $p|r$. Moreover, in the case in Equation \eqref{pure} where $r=p^{2j},$ this value will always be $+1$ if $(d,p)=1$. In contrast, in the general case in Equation \eqref{twistfactor}, the sum may be smaller, due to there being fewer twists where $(d,pr)=1$. 	
		
		Then appealing to Lemma \ref{sufconthm}, we see that the form restricted to powers of a given prime is non-negative definite by Equation \eqref{thmdef}, whilst Equation \eqref{onedom} is ensured by Equation \eqref{thmdom}.	Now we may conclude the proof of Lemma \ref{switchcher}.  		
	\end{proof}
	The proof of Lemma \ref{freebound} is similar, and we follow the same method
	\begin{proof}[Proof of Lemma \ref{freebound}]
		The proof closely follows that of Lemma \ref{switchcher}, and we use much of the same notation. Again, using Lemma \ref{splitbound}, it suffices to reduce to the case where the coefficients are supported on the power of a single prime, $p$. That is, we first must show the diagonal term is non-negative definite and that for any choice of coefficients $\alpha_{p,d_1+2i}$ and $\phi_{p,d_2+2n}$ that
		\begin{align}\nonumber&
			\left|\sum_{m,n\ge 0} \alpha_{p, 2i+d_1}\phi_{p,2n+d_2}\beta_p(p^{2i+2n+d_1+d_2})\right|\le \frac{1}{2}\frac{\beta_p(p^{d_1+d_2})}{\beta_p(p^{2(d_1+d_2)})}\times \\&\sum_{i,j\ge 0} \alpha_{p, d_1+2i}\alpha_{p, d_1+2j} \beta_p(p^{2(d_1+i+j)})+\phi_{p, d_2+2i}\phi_{p, d_2+2j} \beta_p(p^{2(d_2+i+j)}).\label{target}
		\end{align}
		We will then bound this diagonal sum on the right-hand side of \eqref{target} to complete the proof of Lemma \ref{freebound}.
		We begin by calculating the terms $\beta_p(p^{2i+2n+d_1+d_2})$.	
		For the diagonal case, since $\theta_1\ge \theta_2\ge 0$, we can easily check from Equation \eqref{betaf} that $\beta_p(p^0)\ge \beta_p(p^2)\ge 0$ and that $\beta_p(p^{2j})=\beta_p(p^2)\forall j\ge 1$. From here, it immediately follows by Lemma \ref{sufconthm} that the diagonal terms are non-negative definite, and that we have Equation \eqref{target} in the diagonal case where $d_1=d_2.$

		We define \begin{equation}\theta_5=\mathcal{G}_p(1;p^{1+d_1+d_2},v),\end{equation} and note that $\theta_5=1+O\left(\frac{1}{p}\right)$.
		Then \begin{equation}\label{evalbeta}
			\beta_p(p^{2j+d_1+d_2})=(\theta_3-\theta_4)f_j +\theta_4\sum^{\infty}_{i=0}  \frac{a(p)^{2i+\xi_{p}(q_1q_2)}}{p^{i+\frac{\xi_p(q_1q_2)}{2}}(2i)!}-\theta_5\sum^\infty_{i=1}\frac{a(p)^{2i-\xi_{p}(q_1q_2)}}{p^{i-\frac{\xi_p(q_1q_2)}{2}}(2i-1)!}, 
		\end{equation} where \begin{equation}
			f_j=\mathbf{1}(j=0)\frac{a(p)^{d_1+d_2}}{p^{\frac{\xi_p(q_1q_2)}{2}}},
		\end{equation}
		and we recall $\xi_p(q_1q_2)=2\left\{\frac{d_1+d_2}{2}\right\}.$
		In the off-diagonal case where $d_1\ne d_2$, we see $\theta_3=\theta_4,$ and hence $\beta_p(p^{2j+d_1+d_2})=\beta_p(p)$ for all $j$.
		
		By the Cauchy-Schwarz inequality, we hence see
		\begin{align}
			&\nonumber	\sum_{i,j\ge 0} \alpha_{p, 2i+d_1}\phi_{p,2j+d_2}\beta_p(p^{2i+2j+d_1+d_2})\le \frac{1}{2}\beta_p(p)\sum_{i,j\ge 0} \alpha_{p, d_1+2i}\alpha_{p, d_1+2j} +\phi_{p, d_2+2i}\phi_{p,d_2+2j} \\
			&\le \frac{1}{2}\frac{\beta_p(p)}{\beta_p(p^2)}\sum_{i,j\ge 0} \alpha_{p, d_1+2i}\alpha_{p, d_1+2j}\beta_p(p^{2(d_1+i+j)}) +\phi_{p, d_2+2i}\phi_{p,d_2+2j}\beta_p(p^{2(d_2+i+j)}),  \label{interdiag}
		\end{align} 
		where the last inequality follows since only the terms with $d_1+2i=d_1+2j=0$ or\\ $d_2+2i=d_2+2j=0$ will have their coefficients change, and these coefficients will increase. This completes the proof of Equation \eqref{target} in the case $d_1\ne d_2$, bounding the off-diagonal terms in relation to the diagonal terms. 
		
		It now remains to bound the diagonal terms on the right-hand side of Equation \eqref{target}.
		Here, we see that
  \begin{equation}\theta_4=\left(1-\frac{1}{p}\right)\exp\left(O\left(p^{-\frac{3}{2}}\right)\right),\end{equation} and \begin{equation}\label{theta4}
			\theta_3=\begin{cases}
				\theta_4&d_1=1\\
				\theta_1&d_1=0
			\end{cases}.
		\end{equation}
		Returning to Equation \eqref{evalbeta}, we see that the final sum in Equation \eqref{evalbeta} is $\frac{a(p)^2}{p}+O(\frac{1}{p^2}).$ The bound in Equation \eqref{target} depends on the values of $d_1$ and $d_2$; we consider first the bounds where $d_1=d_2.$ When $d_1=0$, when the terms $i=j=0$ have a greater coefficient than when $\max\{i,j\}>0$, or $d_1=1,$ when all the terms have the same coefficient. 			
		If $d_1=0,$ then we see $\theta_1=\exp\left(O\left(p^{-\frac{3}{2}}\right)\right)$. Then Equations \eqref{evalbeta} and \eqref{theta4} yield that \begin{align}
			\beta_p(p^0)&\nonumber= (\theta_1-\theta_4) (1) +\theta_4\left(1+\frac{a(p)^2}{2p}\right)- \frac{a(p)^2}{p}+O\left(\frac{1}{p^2}\right)\\
			&= \exp\left(-\frac{a(p)^2}{2p}\right)\exp\left(O\left(p^{-\frac{3}{2}}\right)\right).
		\end{align}
		
		Meanwhile, if $j\ge 1$ then \begin{align}
			\beta_p(p^{2j})=\beta_p(p^2)=\theta_4\left(1+\frac{a(p)^2}{2p}\right)- \frac{a(p)^2}{p}+O\left(\frac{1}{p^2}\right)&=
			\left(1-\frac{1}{p}\right)\exp\left(-\frac{a(p)^2}{2p}\right)\exp\left(O\left(p^{-\frac{3}{2}}\right)\right).
		\end{align}
		
		Hence, for the case where $d_1=0$, we see 
		\begin{equation}
		\sum_{i,j\ge 0} \alpha_{p, d_1+2i}\alpha_{p, d_1+2j}\beta_p(p^{2(d_1+i+j)}) =\exp(O(p^{-\frac{3}{2}}))\exp\left(-\frac{a(p)^2}{2p}\right)\left(\alpha_{p,0}^2 +\left(1-\frac{1}{p}\right)\left(\sum_{i+j\ge 0} \alpha_{p,i}\alpha_{p,j}\right)\right).
		\end{equation}
				
		It now remains to consider the case $d_1=1$. Here we see that Equations \eqref{evalbeta} and \eqref{theta4} give that, for all $j\ge 0$
		\begin{align}\label{d=0}
		\beta_p(p^{2j+d_1+d_2})=\beta_p(p^2)&=\theta_4 \left(1+\frac{a(p)^2}{2p}\right)-\frac{a(p)^2}{p}+O\left(\frac{1}{p^2}\right)\\
		&=\left(1-\frac{1}{p}\right)\exp\left(-\frac{a(p)^2}{2p}\right)\exp\left(O\left(p^{-\frac{3}{2}}\right)\right).
		\end{align}
		Hence, 
	\begin{equation}\label{d=1}
			\sum_{i,j\ge 0} \alpha_{p, d_1+2i}\alpha_{p, d_1+2j}\beta_p(p^{2(d_1+i+j)}) =\left(1-\frac{1}{p}\right)\exp\left(-\frac{a(p)^2}{2p}\right)\exp\left(O\left(p^{-\frac{3}{2}}\right)\right)\left(\sum_{i}\alpha_{p,1+2i}\right)^2.
	\end{equation}	
	We again remark that the extra factor of $1-\frac{1}{p}$ that multiplies twists at multiples of $p$ comes from the fact that $\chi_d(1)=1$ for all d, whilst for $j\ge 0,$ $\chi_d(p^{2j})=1$ for a proportion $1-\frac{1}{p}$ of the twists, and is $0$ for the remaining twists where $p|d$. 
	
	The combination of Equations \eqref{target} with \eqref{d=0} and \eqref{d=1}, we obtain upper bounds for $$\sum_{i,j\ge 0} \alpha_{p, d_1+2i}\alpha_{p, d_1+2j} \beta_p(p^{2(d_1+i+j)})+\phi_{p, d_2+2i}\phi_{p, d_2+2j} \beta_p(p^{2(d_2+i+j)}),$$ depending on the values of $d_1$ and $d_2$.
	Moreover, if $d_1\ne d_2$, then \begin{equation}\label{offdiagbeta}
		\frac{\beta_p(p^{d_1+d_2})}{\beta_p(p^{2(d_1+d_2)})}=O\left(p^{-\frac{3}{2}}\right),
	\end{equation} while if $d_1=d_2,$ then $\beta_p(p^{d_1+d_2})=\beta_p(p^{2(d_1+d_2)}).$
	Combining these bounds with Lemma \ref{splitbound}, we get bounds on the composite quadratic form at twists over all integers with prime factors lying in $P_j$. Hence, we see \begin{align}
	\nonumber&	\left|\sum_{\substack{p|w_1w_2\implies p\in P_{j}\\ w_1q_1,w_2q_2=\square}}
		\frac{\gamma(w_1)\gamma(w_2)}{(w_1w_2)^{\frac{1}{2}}} \prod_{p\in P_{j}} 	\beta_p(p^{\xi_p(q_1q_2)})\right|\le\frac{1}{2}\prod_{p\in P_j} \exp\left(O\left(p^{-\frac{3}{2}}\right)-\frac{a(p)^2}{2p}\right) \prod_{\substack{r\in P_j\\ \xi_r(q_1q_2)=1}}\frac{\beta_r(r)}{\beta_r(r^2)}\\& \left(\sum_{\substack{p|w_1w_1'\implies p\in P_{j}\\ w_1q_1,w_1'q_1=\square}}
		\frac{\gamma(w_1)\gamma(w_1')}{(w_1w_1')^{\frac{1}{2}}} \prod_{p|w_1w_1'} 	\left(1-\frac{1}{p}\right)+\sum_{\substack{p|w_2w_2'\implies p\in P_{j}\\ w_2q_2,w_2'q_2=\square}}
		\frac{\gamma(w_2)\gamma(w_2')}{(w_2w_2')^{\frac{1}{2}}} \prod_{p|w_2w_2'} 	\left(1-\frac{1}{p}\right)\right).
	\end{align}
	Setting $s_r=\frac{\beta_r(r)}{\beta_r(r^2)}$ and using Equation \eqref{offdiagbeta} completes the proof of Lemma
		 \ref{freebound}. 
	\end{proof}
	\appendix
\section{Moments over quadratic twists}\label{Sec: moments}
In this section, we adapt proofs from \cite{AB24} and \cite{AC25} to the context of moments in the random model of quadratic twists. The following Lemma allows us to take moments of sections of the random walk, inspired by the proof of Lemma 3 in \cite{Sound}. Here we must modify the proof to allow for the diagonal terms including all those with a square product.
\begin{lem}\label{walkmom}
	For any integers $1\le j< k\le R,$ and $r\le {100l_k^2},$ 
	we have 	\begin{equation}
		\E\left[\left|S_{k}-S_{j}\right|^{2r}\right]\ll \frac{(2r)!}{2^rr!}\left(n_k-n_{j}+O\left(\frac{1}{\log X_j}\right)\right)^r. 
	\end{equation}
\end{lem}
\begin{proof}[Proof of Lemma \ref{walkmom}]
	We may write \begin{equation}\left(S_{k}-S_{j}\right)^{r}=\sum_{n\le X_k^r} \frac{b_{j,k,r}(n)\chi_d(n)}{n^{\frac{1}{2}}},\end{equation} where $b_{j,k,r}(n)=0$ unless $n$ is the product of $r$ primes with multiplicity, all lying in $(X_j,X_k].$ In this case, if we can express the prime factorisation as $n=\prod^t_{i=1}p_i^{c_i},$ then \begin{equation}
		b_{j,k,r}(n)={r\choose c_1\dots c_t} \prod^t_{i=1}\frac{a(p_i)^{c_i}}{p_i^{\frac{c_i}{2}}}.
	\end{equation}
	Then by Proposition 2 in \cite{RS2}, we have \begin{align}&\label{walkmoment}
		\sum_{\substack{d\in \mathcal{E}(\mathscr{O},a,v)\\ |d|\le X}}\left|S_{k}-S_{j}\right|^{2r}\Phi\left(\frac{\kappa d}{X}\right)=\check{\Phi}(0)\frac{X}{vN_0}\sum_{\substack{n_1n_2=\square\\ (n_1n_2,v)=1}} 	b_{j,k,r}(n_1)	b_{j,k,r}(n_2)\prod_{p|n_1n_2v}\left(1+\frac{1}{p}\right)^{-1}\prod_{p\nmid N_0} \left(1-\frac{1}{p^2}\right)\\\nonumber
		&			+O\left(\sum_{n_1,n_2} \frac{X^{\frac{1}{2}+\frac{1}{10}}}{|\mathcal{E}|\sqrt{n_1n_2}}	b_{j,k,r}(n_1)	b_{j,k,r}(n_2)\right). 
	\end{align}
	Using the inequality frp, Equation \eqref{CS2}, we can show the error term is $\ll \sum_{n} \frac{X^{\frac{1}{2}+\frac{1}{5}}}{|\mathcal{E}|n}b_{j,k,r}(n)^2.$ It remains to consider the main term. 
	We observe that if $(f_j)^{2t}_{t=1}$ are non-negative integers with $\sum^{2t}_{j=1} f_j=2r,$ then
	\begin{equation}
		{2r\choose f_1\dots f_{2t}}=\sum_{\substack{c_1+...+c_t=r\\
				c_{t+1}+...+c_{2t}=r\\
				c_1,...,c_{2t}\ge 0}}{r\choose c_1\dots c_t}{r\choose c_{t+1}\dots c_{2t}}.
	\end{equation}
	Since all the summands in the main term are non-negative, we see	\begin{align}
&		\sum_{\substack{n_1n_2=\square\\ (n_1n_2,v)=1}} 	b_{j,k,r}(n_1)	b_{j,k,r}(n_2)\prod_{p|n_1n_2v}\left(1+\frac{1}{p}\right)^{-1}\prod_{p\nmid N_0} \left(1-\frac{1}{p^2}\right)\\&\nonumber\le\prod_{p|v} (1+\frac{1}{p})^{-1}\prod_{p\nmid N_0}\left(1-\frac{1}{p^2}\right) \sum_{n} b_{j,k,2r}(n^2).
	\end{align}
	We now have to bound the sum over $n$. We have \begin{align}
		\sum_{n}b_{j,k,2r}(n^2)&=\sum_{c_1+...c_t=r} \frac{(2r)!}{(2c_1)!...(2c_t)!}\left(\frac{a(p_1)^2}{p_1}\right)^{c_1}...\left(\frac{a(p_t)^2}{p_t}\right)^{c_t}\nonumber\\
		&\le \label{sum'} \frac{(2r)!}{r!}\sum_{c_1+...c_t=r}\frac{r!}{(c_1)!...(c_t)!}\left(\frac{a(p_1)^2}{p_1}\right)^{c_1}...\left(\frac{a(p_t)^2}{p_t}\right)^{c_t} \max_{\substack{c'_1+...+c'_t=r\\ c'_1,...,c'_t\ge 0}} \frac{(c'_1)!...(c'_t)!}{(2c'_1)!...(2c'_t)!}.
	\end{align}
	The maximum occurs when $r$ of the $c'_i$ are $1$, and the rest are $0$, when \begin{equation}
		\frac{(c'_1)!...(c'_t)!}{(2c'_1)!...(2c'_t)!}=\frac{1}{2^r}.
	\end{equation}  
	We can now perform the sum in Equation \eqref{sum'}, to show
	\begin{equation}
		\sum_{n}b_{j,k,2r}(n^2)\le \frac{(2r)!}{2^rr!}\left(\sum_{X_j\le p\le X_k}\frac{a(p)^2}{p}\right)^r.
	\end{equation}	
	Returning to Equation \eqref{walkmoment}, and using Lemma 3 from \cite{AB24} to bound the sum over primes, we see the main sum has size \begin{equation}
		\le \frac{X}{vN_0} \frac{(2r)!}{2^rr!}\left(n_k-n_{j}+O\left(\frac{1}{\log X_j}\right)\right)^r. 
	\end{equation}
	
	Moreover, for the error term, if we use Equation \eqref{sum'} to show $\sum_n b_{j,k,r}(n)^2\le \sum_m b_{j,k,2r}(m\nonumber^2)$ and proceed from Equation \eqref{sum'}, we see the error term has size \begin{equation}
		\ll \frac{X^{\frac{1}{
					2}+\frac{1}{5}}}{vN_0|\mathcal{E}|} \frac{(2r)!}{2^rr!}\left(n_k-n_{j}+O\left(\frac{1}{\log X_j}\right)\right)^r, 
	\end{equation} 
	which is negligible.
	
	Thus, we may conclude from Equation \eqref{walkmoment}, that
	\begin{equation}
		\E\left[\left|S_{k}-S_{j}\right|^{2r}\right]= \frac{(2r)!2^r}{r!}\left(n_k-n_{j}+O\left(\frac{1}{\log X_j}\right)\right)^r. 
	\end{equation}
	This completes the proof of Lemma \ref{walkmom}
\end{proof} 
The following Lemma allows us to condition on the value of $S_{i}$ for some $i$, and twist by a factor only depending on primes greater than $X_i$. It is adapted from Lemma 2.4 and 2.7 in \cite{AB}, where the corresponding results for the $t$-aspect are proven.

\begin{lem}\label{pointmom}
	Let $0\le r\le R$, and $w\in [L_r,U_r].$ Suppose $Q_{r+1}$ is a Dirichlet polynomial of the form
	\begin{equation}
		Q_{r+1}=\sum_{\substack{p|m \implies p\in P_{r+1}\\ \Omega_{r+1}(m)\le 10(l_{r+1}-l_{r})^{10^4}}}\frac{\chi_d(m)\gamma(m)}{m^{\frac{1}{2}}},
	\end{equation}
	for some choice of real coefficients $\gamma(m).$
	Then \begin{equation}\label{2.4}
		\E\left[Q_{r+1}^2\mathbf{1}(S_{r}\in [w,w+1])\right]\ll \E[Q_{r+1}^2] \frac{e^{-\frac{w^2}{2n_r}}}{\sqrt{n_r}},
	\end{equation}
	and moreover
	\begin{equation}\label{2.7}
		\E\left[LM_{r+1}Q_{r+1}^2\mathbf{1}(S_{r}\in [w,w+1]\cap G_r)\right]\ll \E[Q_{r+1}^2] \log^{-\frac{1}{2}}X_{r}\frac{e^{-\frac{w^2}{2n_r}}}{\sqrt{n_r}}.
	\end{equation}
\end{lem}
\begin{proof}
	This follows from the proofs of Lemmata 2.4 and 2.7 in \cite{AB}, and also of Lemmata 4.3 and 4.4 in \cite{AC25}; there is essentially no difference in the proof with the orthogonal family of quadratic twists needed for elliptic curves compared to the $t$-aspect or the $q$-aspect. Since Proposition \ref{twistmollprop} a twisted mollifier formula for the well-factorable twists that allows for twists of length up to $X^{\frac{1}{1000}},$ we can twist by the necessary Dirichlet polynomials used to approximate the indicator function of the walk $S_{i}$ lying in small subintervals of $[L_i,U_i]$ for $0\le i\le r.$
\end{proof}		
\section{Bounds on quadratic forms}\label{Sec: QF}
It is often convenient to be able to split bounds of cross terms in twisted moments into bounds on prime powers. The following Lemma provides a way to split bounds on cross terms, by viewing them as quadratic forms. 
\begin{lem}\label{splitbound} 
	Let $n$ be a positive integer, and for $1\le j\le n,$ let $V_j$ be a real (resp. complex) finite-dimensional vector spaces, and  $Z_j$  and $R_j$  be a
	symmetric (resp. Hermitian) quadratic forms on $V_j$. Suppose that $Z_j$ is non-negative definite, and that for all $1\le j\le n$, and $\underline{\alpha}_j, \underline{\phi}_j\in V_j$ we have:
	
	\begin{equation}\label{onedom}\left| R_j(\underline{\alpha}_j,\underline{\phi}_j)\right|\le \frac{1}{2}\left(Z_j(\underline{\alpha}_j,\underline{\alpha}_j)+Z_j(\underline{\phi}_j,\underline{\phi}_j)\right).\end{equation}
	
	Then $Z:=\otimes_j Z_j$ and $R:=\otimes_j R_j$ are  
	symmetric (resp. Hermitian) quadratic forms,  on $V:=\otimes_j V_j$, with $Z$ non-negative definite. Moreover, for all $\underline{\alpha}$ and $\underline{\phi}$ in $V$, we have 
	\begin{equation}\label{mixdom}
		\left| R(\underline{\alpha},\underline{\phi})\right|\le \frac{1}{2}\left(Z(\underline{\alpha},\underline{\alpha})+Z(\underline{\phi},\underline{\phi})\right).\end{equation}
\end{lem}

\begin{proof}[Proof of Lemma \ref{splitbound}] 
	
	Note that $Z$ is clearly non-negative definite, as the tensor product of non-negative definite quadratic forms. By performing a small perturbation of the $Z_j$, we may assume  that each $Z_j$ is positive definite. We can pick bases   for each vector space $V_j$, and matrices $A_j$ and $B_j$ with respect to the bases such that
	\begin{equation}
		Z_j(\underline{\alpha}_j,\underline{\phi}_j)={^\dagger\underline{\alpha}_j}A_j\underline{\phi}_j, \quad R_j(\underline{\alpha}_j,\underline{\phi}_j)={^\dagger\underline{\alpha}_j}B_j\underline{\phi}_j, \quad 
	\end{equation}
	Since $Z_j$ is symmetric (resp. Hermitian), we can find invertible symmetric (resp. Hermitian)  matrices $C_j$ such that \begin{equation}	
		A_j=C_j^2.\end{equation}
	
	Then $C_j^{-1}B_j C_j^{-1}$ is a symmetric (resp. Hermitian) matrix, and hence has an eigenbasis of $V_j$ with real eigenvalues $(\lambda^{(i_j)}_j).$
	
	The boundedness condition in Equation \eqref{onedom} is equivalent to: \begin{equation}\label{evalbound}
		|\lambda^{(i_j)}_j|\le 1\forall i.\end{equation}
	
	But with respect to the product basis of $V$ formed from the eigenbases, the quadratic $Z$ and $R$ are represented by $A:=\otimes_j A_j$ and $B:=\otimes_j B_j$ respectively, and $C:=\otimes_j C_j$ is  a symmetric (resp. Hermitian) matrix, such that \begin{equation}		 				
		A=C^2.\end{equation}
	
	Since $C^{-1}BC$ is  a symmetric (resp. Hermitian) matrix with eigenvalues $\left(\prod_j \lambda_{j}^{(i_j)}\right)$, and $|\prod_j\lambda_{j}^{(i_j)}|\le 1$ for any choice of eigenvalues by Equation $\eqref{evalbound},$ we see that
	
	\begin{equation}
		\left| R(\alpha,\phi)\right|\le \frac{1}{2}\left(Z(\alpha,\alpha)+Z(\phi,\phi)\right),\end{equation}
	as required.	
\end{proof}
In many applications of Lemma \ref{splitbound}, the quadratic forms will correspond to twists at powers of an individual prime $p$, and we will exploit that $\chi_d(p^{2+u})=\chi_d(p^u)$ whenever $u\ge 1$. The following Lemma gives a convenient form for proving the conditions on Equation \eqref{onedom}. 
\begin{lem}\label{sufconthm}
Let $V$ be a real vector space with basis $B$, and $v_0$ be a given element in $B$. Suppose we have continuous quadratic forms $Z$ and $R$ on $V$ and real values $\theta_1\ge \theta_2\ge 0$ and $\theta_3\ge \theta_4\ge 0$. Moreover, suppose that $\theta_2\ge \theta_4, \theta_1-\theta_2\ge \theta_3-\theta_4$ and that whenever $u,v\in B$ we have
\begin{equation}
		Z(u,v)=\begin{cases}
			\theta_1& u=v=v_0\\
			\theta_2& \text{otherwise}
		\end{cases}, \quad 
	R(u,v)=\begin{cases}
		\theta_3& u=v=v_0\\
		\theta_4& \text{otherwise}
	\end{cases}.\end{equation}
	Then for any choice of $x,y\in V$ we have \begin{equation}\label{thmdef}
		Z(x,x)\ge 0
	\end{equation}
	and \begin{equation}\label{thmdom}
		|R(x,y)|\le \frac{1}{2}\left(Z(x,x)+Z(y,y)\right)
	\end{equation}
\end{lem} 
\begin{proof}[Proof of Lemma \ref{sufconthm}]
	We may define the coefficients of the decomposition into the basis as \begin{equation}x=\sum_{v\in B}\lambda_v v,\quad y=\sum_{v\in B}\mu_v v.\end{equation}
	Then by continuity, we can calculate the quadratic forms on the basis elements, to write
	\begin{align}
	Z(x,x)&=\sum_{u,v\in B} Z(\lambda_u u,\lambda_v v)\\
	&= \lambda_{v_0}^2 \theta_1+\sum_{(u,v)\in B^2\setminus\{(v_0,v_0)\}}\lambda_u \lambda_v \theta_2\\
		&= (\theta_1-\theta_2)\lambda_{v_0}^2 +\theta_2 \left(\sum_{v\in B} \lambda_v\right)^2\ge 0,
	\end{align}
	which completes the proof of Equation \eqref{thmdef}.
	Similarly, we see \begin{align}
		|R(x,y)|&=\left|(\theta_3-\theta_4) \lambda_{v_0}\mu_{v_0}+\theta_4\left(\sum_{u\in B} \lambda_u\right)\left(\sum_{v\in B} \lambda_v \right)\right|\\&\le
		\frac{1}{2}\left((\theta_3-\theta_4)\left(\lambda_{v_0}^2+\mu_{v_0}^2\right)+\theta_4\left(\left(\sum_{u\in B} \lambda_u\right)^2+\left(\sum_{v\in B} \mu_v\right)^2\right)\right), 
 	\end{align}
	where the last line followed by two applications of the Cauchy-Schwarz inequality.
	Using the inequalities between the $\theta_j,$ we complete the proof of Equation \eqref{thmdom}.
\end{proof}

\bibliographystyle{alpha}
\bibliography{Upper_bounds_for_large_deviations_elliptic_curves.bib}
\end{document}